\documentclass[twoside, final]{article}

\title{Ergodicity in Randomly Forced Rayleigh-B\'{e}nard Convection}
\author{J. F\"{o}ldes, N. Glatt-Holtz, G. Richards, J.P. Whitehead\\
  \scriptsize{emails:  juraj.foldes@ulb.ac.be, negh@vt.edu, g.richards@rochester.edu, whitehead@mathematics.byu.edu}}

\usepackage{amsfonts, amssymb, amsmath, amsthm,enumerate,esint}
\usepackage[colorlinks=true, pdfstartview=FitV, linkcolor=blue,
            citecolor=blue, urlcolor=blue]{hyperref}
\usepackage[usenames]{color}
\definecolor{Red}{rgb}{0.7,0,0.1}
\definecolor{Green}{rgb}{0,0.7,0}
\usepackage{accents}
\usepackage{comment}

\usepackage[notref,notcite,color]{showkeys}
\usepackage{bbm}
\definecolor{labelkey}{rgb}{0,0,1}

\numberwithin{equation}{section}

\newtheorem{Thm}{Theorem}[section]

\newtheorem{Prop}[Thm]{Proposition}
\newtheorem{Cor}[Thm]{Corollary}

\newtheorem{Rmk}{Remark}[section]

\newcommand{\pd}[1]{\partial_{#1}}

\newcommand{\E}{\mathbb{E}}
\newcommand{\Prb}{\mathbb{P}}
\newcommand{\RR}{\mathbb{R}}

\newcommand{\bfU}{\mathbf{u}}
\newcommand{\bfV}{\mathbf{v}}
\newcommand{\T}{T}
\newcommand{\Tdiff}{\phi}
\newcommand{\prN}{Pr}
\newcommand{\Nu}{Nu}

\newcommand{\hate}{\hat{\mathbf{e}}_{d}}

\newcommand{\DD}{\mathcal{D}}
\newcommand{\Trs}{\theta}
\newcommand{\Utot}{U}
\newcommand{\Vtot}{V}
\newcommand{\NN}{\mathbb{N}}
\newcommand{\Uo}{\mathbf{u}^{0}}

\newcommand{\tstar}{t^*}
\newcommand{\vstar}{{\bf{{v}}}^*}
\newcommand{\Tstar}{{T^*}}
\newcommand{\pstar}{p^*}
\newcommand{\sigmastar}{\sigma^*}
\newcommand{\sigmatildestar}{\tilde{\sigma}^*}
\newcommand{\Wstar}{W_*}
\newcommand{\tildeWstar}{{\tilde{W}}_*}
\newcommand{\astar}{\alpha}

\newcommand{\rab}{\tilde{Ra}}
\newcommand{\ra}{Ra}

\newcommand{\Tho}{\theta^{0}}

\begin{document}
\markboth{J. F\"{o}ldes, N. Glatt-Holtz, G. Richards, J.P. Whitehead}
{Ergodicity in Randomly Forced Rayleigh-B\'{e}nard Convection}

\maketitle

\begin{abstract}
We consider the Boussinesq approximation for Rayleigh-B\'{e}nard convection perturbed by an additive noise
and with boundary conditions corresponding to heating from below.
In two space dimensions, with sufficient stochastic forcing in the temperature component and large
Prandtl number $\prN> 0$, we establish the existence of a unique ergodic invariant measure.
In three space dimensions, we prove the existence of a statistically invariant state, and establish unique ergodicity
for the infinite Prandtl Boussinesq system.  Throughout this work we provide streamlined proofs of unique ergodicity which invoke
an asymptotic coupling argument, a delicate usage of the maximum principle, and exponential martingale inequalities.
Lastly, we show that the background method of Constantin-Doering
\cite{ConstantinDoering1996} can be applied in our stochastic setting, and prove
bounds on the Nusselt number relative to the unique invariant measure.
\end{abstract}

\setcounter{tocdepth}{1}
\tableofcontents

\section{Introduction}

Following experiments of B\'{e}nard \cite{Benard1900}, Rayleigh \cite{Rayleigh1916} proposed the equations of Boussinesq
\cite{Boussinesq1897} as an effective model for the flow of a fluid driven by buoyancy forces due to heating from below
and cooling from above, a phenomenon now referred to as Rayleigh-B\'{e}nard convection.
These equations have since appeared in a wide variety of physical models,
including descriptions of climate and weather processes and the internal dynamics of both planets and stars.

Individual solutions of the Boussinesq system can be unpredictable and seemingly chaotic,
particularly in parameter ranges leading to turbulent regimes.  However, some of the statistical properties of solutions are robust.
It is therefore of fundamental significance to identify and predict statistical features of Rayleigh-B\'{e}nard convection, and to
connect these features to rigorous theory at the level of the Boussinesq equations.  Indeed, the fine scale structure of flows,
complex pattern formation, and the mean heat transport, for example, remain topics of intensive theoretical, numerical and experimental research.
Here the role of analysis is particularly significant in parameter ranges beyond the capacity of direct
numerical simulation or empirical observability.  See \cite{BodenschatzPeschAhlers2000, Manneville2006, AhlersGrossmannLohse2009, LohseXia2010} for
a survey of recent developments in the physics literature.

A mathematically rigorous theory of statistical properties should include the analysis of invariant measures for the system,
which contain important statistics of the flow.  Natural questions include the existence, uniqueness, ergodicity and other
attraction properties of invariant measures.  Moreover one may seek
to prove quantitative bounds on statistical quantities determined by flows in terms of these measures.

There is a significant literature devoted to proving rigorous quantitative bounds for the Boussinesq equations,
primarily focused on estimating rates of convective heat transport.
 This direction of research was initiated by \cite{malkus1954,howard1972,busse1970,busse1978}, advanced significantly
with the invention of the ``background flow method'' \cite{ConstantinDoering1996,ConstantinDoerin1999,ConstantinDoering2001} (see also \cite{hopf1940}),
and refined in more recent works (e.g. \cite{WhiteheadDoering2011,OttoSeis2011}).
It is noteworthy that practical methodologies for proving rigorous bounds on key statistical quantities have not been identified
in a stochastic setting.

On the other hand, while some works have established existence and convergence properties of invariant measures for
the Boussinesq system \cite{Wang2008b}, it is difficult, in general, to obtain uniqueness or erdogicity results for systems of
deterministic partial differential equations (cf. \cite{FoiasManleyRosaTemam01}).  These problems become more tractable by including a stochastic forcing,
due to smoothing properties of the corresponding probability distribution functions induced by random perturbations in the equations.
Moreover, as early as the 19th century, Boussinesq conjectured that turbulent flow cannot be described solely with deterministic methods,
and indicated that a stochastic framework should be used \cite{Stanisic1985}.  This setting is now ubiquitous in the
turbulence literature, see e.g. \cite{Novikov1965, VishikKomechFusikov1979, Eyink96} and containing references.
In particular, note that some works have considered stochastic initial and boundary conditions for the Boussinesq system to predict qualitative features of the flow,
including the onset of turbulence \cite{VenturiWanKarniadakis10,VenturiChoiKarniadakis12}.

In the manuscript we consider a stochastic Boussinesq system
\begin{align}
  \frac{1}{\prN}&(d \bfU + \bfU \cdot \nabla \bfU dt) + \nabla p dt = \Delta \bfU dt + \ra \hate \T dt + \sum_{k = 1}^{N_1} \tilde{\sigma}_k d\tilde{W}^k, \quad \nabla \cdot \bfU = 0,
  \label{eq:B:eqn:vel}\\
  &d \T + \bfU \cdot \nabla \T dt = \Delta \T dt +  \sum_{k = 1}^{N_2}\sigma_k dW^k ,
  \label{eq:B:eqn:temp}
\end{align}
for the (non-dimensionalized) velocity field $\bfU = (u_1,\ldots,u_d)$ (where $d=2$ or $d=3$),
pressure $p$, and temperature $T$
of a buoyancy driven fluid.
The system \eqref{eq:B:eqn:vel}--\eqref{eq:B:eqn:temp}
evolves in a domain $\mathbf{x}=(x_1,\ldots,x_d)\in\DD = [0,L]^{d-1} \times [0,1]$ and is supplemented with boundary conditions to be specified below.
Here $\hate=(0,\ldots,1)$ is a unit vector pointing in the vertical direction.

The driving noise is given by a collection of independent white noise processes $d\tilde{W}^k = d\tilde{W}^k(t)$ and $dW^k = dW^k(t)$ acting in spatial
directions $\tilde{\sigma}_k = \tilde{\sigma}_k(x)$, $\sigma_k = \sigma_k(x)$ which
form a complete orthogonal basis of eigenfunctions (ordered with respect to eigenvalues) of the Stokes and Laplace operators on $\DD$, respectively,
with appropriate boundary conditions.
\footnote{Our analysis does not require stochastic perturbation in a
diagonal basis of eigenfunctions (we assume this form for simplicity of presentation).  More generally, we could consider any
$\{\tilde{\sigma}_k\}$ and $\{\sigma_k\}$ which span a determining set of directions for the governing equations.}
  The number of forced modes, $N_1$ and $N_2$, are both finite and will be further specified in theorem statements below.
We can treat the case where the numbers $N_1$ or $N_2$ are infinite,
provided we impose enough decay in the bases $\{\tilde{\sigma}_k\}$ and $\{\sigma_k\}$ such that the system \eqref{eq:B:eqn:vel}--\eqref{eq:B:eqn:temp}
remains globally well-posed according to Propositions \ref{prop:existence-uniqueness:2d}--\ref{prop:existence-uniqueness:3d}
(see Remark \ref{rem:inf} below).
This does not complicate our analysis in a significant way, and for simplicity of
presentation, we assume $N_1,N_2<\infty$ unless stated otherwise.

We will consider, in particular, the case $N_1=0$; that is, \eqref{eq:B:eqn:vel}--\eqref{eq:B:eqn:temp} with no stochastic forcing in the
velocity component.  This is partly motivated by investigations of the (determistic) Boussinesq system which have considered convection driven by internal heating
(see \cite{Roberts1967,TrittonZarraga1967,LuDoeringBusse2004, WhiteheadDoering2011,GoluskinSpiegel2012, BarlettaNield2012}), to describe,
for example, radioactive decay processes in the earth's mantle.

In previous work of the first three authors in collaboration with Thomann \cite{FoldesGlattHoltzRichardsThomann2013}, we established ergodic and mixing properties
for \eqref{eq:B:eqn:vel}--\eqref{eq:B:eqn:temp} in the two-dimensional periodic domain (i.e. $\DD = \mathbb{T}^2$) with a stochastic forcing acting on
a small collection of low frequency modes in the temperature component only; that is, with $N_1=0$ and small $N_2>0$.  This form of spatially degenerate
random forcing (i.e. $N_2$ small) is motivated by the turbulence literature, where it is conjectured that nonlinear terms will propagate excitation
to higher frequencies, and the system will converge to a unique statistical equilibrium (see \cite{Novikov1965, VishikKomechFusikov1979, Eyink96}).
The results of \cite{FoldesGlattHoltzRichardsThomann2013} generalized recent progress of \cite{HairerMattingly06,HairerMattingly2008,HairerMattingly2011}
on the stochastic Navier-Stokes equations and related systems.  Indeed, note that with $N_1=0$ we are not forcing the velocity field in
\eqref{eq:B:eqn:vel}--\eqref{eq:B:eqn:temp} directly, which is a more degenerate setting than was considered in
\cite{HairerMattingly06,HairerMattingly2008,HairerMattingly2011}.

From the physical point of view, using periodic boundary conditions in the vertical direction for \eqref{eq:B:eqn:vel}--\eqref{eq:B:eqn:temp} is not
 appropriate.  Instead, one should fix the temperature on the upper and lower boundaries (corresponding to heating from below), and employ Dirichlet conditions
 in the velocity field, as follows
\begin{align}
   \bfU_{|x_d = 0} = \bfU_{|x_d = 1} = 0, \quad T_{|x_d = 0} = \rab, \;  T_{|x_d = 1} = 0, \quad
   \bfU, T \textrm{ are periodic in } \mathbf{x} = (x_1,\ldots, x_{d-1}).
   \label{eq:bc}
\end{align}
The positive unitless physical parameters in the problem \eqref{eq:B:eqn:vel}--\eqref{eq:B:eqn:temp} with boundary conditions
\eqref{eq:bc} are the \emph{Prandtl} number $\prN$ and \emph{Rayleigh} numbers $\ra$ and $\rab$; see Section~\ref{sec:non-dim:equations}
below for further details.  Our first objective is to establish existence and uniqueness properties of invariant measures for
\eqref{eq:B:eqn:vel}--\eqref{eq:B:eqn:temp} with boundary conditions given by \eqref{eq:bc}.

\begin{Thm}\label{thm:1}
Suppose the spatial dimension is $d=2$. Then the system \eqref{eq:B:eqn:vel}--\eqref{eq:B:eqn:temp} with boundary conditions \eqref{eq:bc}
possesses a unique ergodic invariant probability measure if at least one of the following holds:
\begin{enumerate}[(i)]
\item  $N_{1}=N_1(\prN,\ra,\rab)>0$ and $N_{2}=N_{2}(\prN,\ra,\rab)>0$ are both sufficiently large.
\item
 $\prN>0$ and $N_{2}=N_2(\prN,\ra,\rab)>0$ are both sufficiently large.
\end{enumerate}
\end{Thm}

For the proof of Theorem \ref{thm:1}, we use a streamlined argument which may be of broader interest in the theory of ergodicity for
infinite-dimensional systems.  A similar approach is used in concurrent work of two of the authors in collaboration with Mattingly
\cite{GlattHoltzMattinglyRichards2015}, where this technique has been implemented for a number of other nonlinear stochastic PDEs.
The argument invokes an abstract framework developed for application to SDEs with delay \cite{HairerMattinglyScheutzow2011},
allowing us to significantly reduce the length and technical detail of the proofs.

More precisely, by applying a theorem of \cite{HairerMattinglyScheutzow2011}
(see Theorem \ref{thm:HMS2011}), we can reduce the problem of uniqueness to the convergence of solutions of \eqref{eq:B:eqn:vel}--\eqref{eq:B:eqn:temp}
to solutions of a shifted system (see \eqref{eq:B:eqn:vel:sft}--\eqref{eq:B:eqn:temp:sft} below).  In order to apply the result of
\cite{HairerMattinglyScheutzow2011}, we invoke the Girsanov theorem to establish the equivalence of (the laws of) solutions to the shifted system to
those of \eqref{eq:B:eqn:vel}--\eqref{eq:B:eqn:temp}, and prove the desired convergence at time infinity by using Foias-Prodi type bounds, a stopping
time argument, and a priori estimates on solutions.  While the basic ingredients of this method are standard tools in the field
(e.g. see \cite{KuksinShirikian12}), we believe that its brevity and simplicity makes it useful.
We emphasize that in the context of the Boussinesq system,
the a priori estimates on solutions will be proven with a nontrivial comparison argument
invoking the maximum principle and certain exponential martingale inequalities (see Section \ref{sec:3} for more details).

In three space dimensions, we prove the existence of statistically invariant states by means of a regularization along with suitable a priori estimates
following the general strategy from \cite{FlandoliGatarek1} (see also \cite{GlattHoltzSverakVicol2013}).
Here our approach is to consider a Galerkin truncation imposed only in the velocity equations \eqref{eq:B:eqn:vel}.
As such we are able to preserve the advection diffusion structure of \eqref{eq:B:eqn:temp}, which plays a critical role in the analysis.

\begin{Thm}\label{thm:1b}
Suppose the spatial dimension is $d=3$, then the system \eqref{eq:B:eqn:vel}--\eqref{eq:B:eqn:temp} with boundary conditions \eqref{eq:bc} possesses at least one statistically invariant state.
\end{Thm}

The analysis of convection in the large Prandtl number limit is relevant in numerous
contexts, such as modeling of the earth's mantle and for convection in high pressure gasses,
where $\prN$ can reach the order of $10^{24}$ (see \cite{ConstantinDoerin1999, ConstantinDoering2001, OttoSeis2011}).
Taking $N_1=0$, and substituting $\prN = \infty$ into \eqref{eq:B:eqn:vel}--\eqref{eq:B:eqn:temp},
we formally obtain the stochastic infinite Prandtl Boussinesq system
\begin{align}
   &- \Delta \bfU =  \nabla p  + \ra \hate \T, \quad \nabla \cdot \bfU = 0,
  \label{eq:B:vel:inf}\\
  &d \T + \bfU \cdot \nabla \T dt = \Delta \T dt +  \sum_{k=1}^{N_2} \sigma_k dW^k,
  \label{eq:B:temp:inf}
\end{align}
complemented with boundary conditions for $\T$ and $\bfU$ as in \eqref{eq:bc}.  Note that \eqref{eq:B:vel:inf}--\eqref{eq:B:temp:inf} is an active scalar equation for the temperature $\T$, and the velocity field $\bfU$ is enslaved to $\T$.
We remark that the system \eqref{eq:B:vel:inf}--\eqref{eq:B:temp:inf} can have complex dynamics, even without
stochastic forcing, provided the Rayleigh number $\ra$ is sufficiently large; see \cite{BreuerHansen09,ConstantinDoerin1999,
BodenschatzPeschAhlers2000, ConstantinDoering2001, Wang2004a, park2006,
AhlersGrossmannLohse2009, LohseXia2010, OttoSeis2011}.

In a companion work of the first three authors \cite{FoldesGlattHoltzRichardsThomann2014b}, we have recently established uniqueness and mixing properties of the invariant probability measure
for the system \eqref{eq:B:vel:inf}--\eqref{eq:B:temp:inf}.  In this manuscript we present a more direct proof of uniqueness in order
to highlight another application of the simplified method from \cite{HairerMattinglyScheutzow2011,GlattHoltzMattinglyRichards2015}.
\begin{Thm}
Suppose the spatial dimension is $d=2$ or $d=3$.
If $N_2=N_2(\ra,\rab)>0$ is sufficiently large, then the system \eqref{eq:B:vel:inf}--\eqref{eq:B:temp:inf} possesses a unique ergodic invariant probability
measure.
\label{thm:2}
\end{Thm}
In \cite{FoldesGlattHoltzRichardsThomann2014b} we also studied asymptotics in the infinite Prandtl limit.
Namely, we showed that as $\prN \rightarrow \infty$, statistically invariant states of \eqref{eq:B:eqn:vel}--\eqref{eq:B:eqn:temp}
(which exist by Theorems \ref{thm:1}--\ref{thm:1b}) converge weakly (in the temperature component) to the unique invariant measure of \eqref{eq:B:vel:inf}--\eqref{eq:B:temp:inf}.
The proof was based on establishing that the Markovian dynamics of \eqref{eq:B:vel:inf}--\eqref{eq:B:temp:inf} are
contractive with respect to an appropriate Kantorovich-Wasserstein metric.  Using this contraction property, we reduced the question of weak
convergence of invariant states as $\prN \rightarrow \infty$ to one of (fixed) finite time asymptotics, and $\prN$-uniform exponential moment
bounds on invariant states of \eqref{eq:B:eqn:vel}--\eqref{eq:B:eqn:temp} and \eqref{eq:B:vel:inf}--\eqref{eq:B:temp:inf}.  In this paper
we include complete proofs of the $\prN$-uniform bounds required in the analysis of \cite{FoldesGlattHoltzRichardsThomann2014b}
(see Sections \ref{sec:3} and \ref{sec:ex}, and specifically Corollaries \ref{cor:ref1} and \ref{cor:2} below).
The proofs of these estimates invoke a comparison argument based on the maximum principle and weighted bounds due to coupling in the systems
\eqref{eq:B:eqn:vel}--\eqref{eq:B:eqn:temp} and \eqref{eq:B:vel:inf}--\eqref{eq:B:temp:inf}.

For the final observation of this manuscript, we illustrate that the background method of Constantin and Doering
\cite{ConstantinDoering1996,ConstantinDoerin1999,ConstantinDoering2001} applies to the stochastic Boussinesq system
\eqref{eq:B:eqn:vel}--\eqref{eq:B:eqn:temp}.  The Nusselt number $\Nu$, defined as a long-time average, is the ratio of the convective to the
conductive heat transport.  The background method provides an upper bound on the Nusselt number as a function of the
strength of the forcing mechanisms in the system (through both boundary and body forcing in the temperature equation), to illustrate an inherent restriction
on the convection produced by heat sources in Boussinesq flows.
In our setting, for $\prN$ large, we will obtain rigorous bounds on the Nusselt number $\Nu$ for \eqref{eq:B:eqn:vel}--\eqref{eq:B:eqn:temp}
in terms of $\ra$ and $\rab$.

More precisely, for $d=2$, $N_1=0$ and $N_2=\infty$
(with sufficient decay in the basis functions $\{\sigma_k\}$ so that \eqref{eq:B:eqn:vel}--\eqref{eq:B:eqn:temp}
is globally well-posed, see Remark \ref{rem:inf}), we define the Nusselt number (relative to $\mu$) by
\footnote{In this context we define the conductive heat transport as $\rab$, the temperature difference between the plates.}
\begin{align}
\label{eq:nu:0}
\Nu := \frac{1}{\rab|\mathcal{D}|}\int\int_{\mathcal{D}}(u_2 T - \pd{2}T)\,dx \,d\mu(\bfU,T) = 1 + \frac{1}{\rab|\mathcal{D}|}\int
\int_{\mathcal{D}}u_2 T\, dx\,d\mu(\bfU,T),
\end{align}
where $\mu$ is the unique invariant measure of the system \eqref{eq:B:eqn:vel}--\eqref{eq:B:eqn:temp} (by Theorem \ref{thm:1}),
and $\pd{2}$ denotes the partial derivative in the $x_2$ direction.
Here the integral involving $\mu$ is taken over an appropriate phase space for \eqref{eq:B:eqn:vel}--\eqref{eq:B:eqn:temp}, see
Section \ref{sec:5} for details.
Let us emphasize that we have defined the Nusselt number as a
statistical average against the unique invariant measure $\mu$, rather than using long time averages of the flow,
which is a more standard interpretation.
However, by invoking the ergodicity from Theorem \ref{thm:1},
 we can equate these two definitions, and apply the background method to obtain quantitative bounds on $\Nu$.
We will reproduce the simplest case of \cite{ConstantinDoering1996} which gives $\Nu \lesssim (\ra\rab)^{1/2}$ for $\ra,\rab \gg 1$ in our context,
but it appears that one could adapt other arguments to produce sharper bounds (see e.g. \cite{OttoSeis2011}).

\begin{Thm}
\label{thm:3} Suppose $d=2$, $N_1=0$, $N_2=\infty$, and assume that $\prN>0$ is sufficiently large such that, by Theorem \ref{thm:1}, the system \eqref{eq:B:eqn:vel}--\eqref{eq:B:eqn:temp} possesses
unique ergodic invariant measure $\mu$.  Then the Nusselt number $\Nu$ given by \eqref{eq:nu:0} satisfies:
\begin{enumerate}[(i)]
\item  For $\mu$-almost every initial condition $(\bfU_0,T_0)\in H$,
\begin{align*}
\Nu = \lim_{t\rightarrow \infty}\frac{1}{\rab |\mathcal{D}|}\E\left(\frac{1}{t}\int_{0}^{t}\int_{\mathcal{D}}(u_2 T - \pd{2}T)(x,s)dx ds\right),
\end{align*}
where $(\bfU,T)$ is the solution of \eqref{eq:B:eqn:vel}--\eqref{eq:B:eqn:temp} with initial data $(\bfU_0,T_0)$.
\item  $\Nu \leq C(\ra\rab)^{1/2}$ for $\ra,\rab>0$ large, where $C=C(|\mathcal{D}|)>0$.
\footnote{Note that $\ra\rab = \frac{g\alpha h^3 T_1}{\nu\kappa}$, which is the usual Rayleigh number in deterministic convection problems.}
\end{enumerate}
\end{Thm}

The remainder of this manuscript is organized as follows.  In Section \ref{sec:2} we provide the precise mathematical framework
of the manuscript.  In Section \ref{sec:3} we establish a priori estimates on solutions to \eqref{eq:B:eqn:vel}--\eqref{eq:B:eqn:temp},
\eqref{eq:B:vel:inf}--\eqref{eq:B:temp:inf}, and a larger class of stochastic drift-diffusion equations.  We discuss the existence of invariant states in
Section \ref{sec:ex}, and present the proof of Theorem \ref{thm:1b}.  Section \ref{sec:4} is devoted to the proofs
of our main unique ergodic theorems, Theorem \ref{thm:1} and Theorem \ref{thm:2}.  In Section \ref{sec:5} we show how the background method adapts to our stochastic setting, and
provide the proof of Theorem \ref{thm:3}.  Lastly, we include an appendix (Appendix \ref{sec:non-dim:equations}) which provides some details of rescaling arguments
for our model equations.

\section{Mathematical Framework}
\label{sec:2}
In this manuscript we study the stochastic Boussinesq equations \eqref{eq:B:eqn:vel}--\eqref{eq:B:eqn:temp} and \eqref{eq:B:vel:inf}--\eqref{eq:B:temp:inf}.
As usual, these equations are rigorously understood in a time integrated sense.
The unitless physical parameters in the problem are the \emph{Prandtl number}, $\prN$, and the
\emph{Rayleigh numbers}, $\ra$ and $\rab$.  The system \eqref{eq:B:eqn:vel}--\eqref{eq:B:eqn:temp} is the result of a rescaling.  In fact,
\begin{align}
\label{eq:phys}
  \prN = \frac{\nu}{\kappa}, \quad \ra=  \frac{g \alpha  \gamma h^{4-d/2}}{\nu \kappa^{3/2}}, \quad
 	\rab = \frac{\sqrt{\kappa}h^{d/2-1} T_1}{\gamma}\,,
\end{align}
where $h$ represents the height of the domain, $T_1$ the applied temperature difference,
and $\gamma$ is a stochastic heat flux representing the strength of the random forcing in the temperature component
of the original variables.
Also $\nu$ is the kinematic viscosity, $\kappa$ is the coefficient of thermal diffusivity,
$g$ is the gravitational constant and  $\alpha$ is the coefficient of thermal expansion.
The orthogonal basis $\{\sigma_k\}$ appearing in the stochastic terms of \eqref{eq:B:eqn:temp} have been normalized,
depending on the number of forced modes, such that for a given fixed $N_2>0$,
\begin{align*}
\|\sigma\|^{2}:=\sum_{k=1}^{N_2} \|\sigma_k\|_{L^2(\DD)}^2 = 1 \,,
\end{align*}
with the strength of the body forcing expressed through $\ra$ and $\rab$.
In particular we always assume forcing in the temperature equation, but we
prescribe no normalization condition on forcing in the velocity equation.
We provide more details on the formulation of our system in Appendix \ref{sec:non-dim:equations} below.

We will often subtract a linear profile from $T$ in order to replace \eqref{eq:B:eqn:temp} with
a system satisfying homogeneous boundary conditions.  Define
\begin{align}
\label{def:theta}
  \Trs := T - \rab(1 -x_d),
\end{align}
we obtain
\begin{align}
\frac{1}{\prN}&(d \bfU + \bfU \cdot \nabla \bfU dt) + \nabla p dt = \Delta \bfU dt + \ra \hate \Trs dt
+ \sum_{k = 1}^{N_1} \tilde{\sigma}_k d\tilde{W}^k, \quad \nabla \cdot \bfU = 0,
  \label{eq:B:eqn:vel:2}\\
  &d \Trs + \bfU \cdot \nabla \Trs dt = \rab u_d dt + \Delta \Trs dt + \sum_{k =1}^{N_2} \sigma_k dW^k,
  \label{eq:theta}
\end{align}
with the boundary conditions modified on the physical boundary as
\begin{align}
   \bfU_{|x_d = 0} = \bfU_{|x_d = 1} = 0, \quad \Trs_{|x_d = 0} =  \Trs_{|x_d = 1} = 0,
   \label{eq:bc:3}
\end{align}
and initial conditions
$\Trs(t=0)=\Trs_0=T_0-\rab(1-x_d)$.
Notice that we have implicitly modified the pressure in \eqref{eq:B:eqn:vel:2} by
$\rab(x_d-\frac{1}{2}x_d^2)$ since $(1-x_d)\hate=\nabla (x_d-\frac{1}{2}x_d^2)$.

We will also consider the infinite-Prandtl Boussinesq system
\begin{align}
 &(-\Delta \Uo + \nabla p) dt =  \ra \hate \Tho dt, \quad \nabla \cdot \Uo = 0,
  \label{eq:vel:inf}\\
\label{eq:theta:inf}
  &d \Tho + \Uo \cdot \nabla \Tho dt = \rab u^{0}_d dt + \Delta \Tho dt + \sum_{k =1}^{N_2} \sigma_k dW^k
\end{align}
complemented with
\begin{align*}
\Tho_{|x_d = 0} =  \Tho_{|x_d = 1} = 0,
\end{align*}
and periodic boundary conditions in the horizontal directions.

For much of the analysis that follows, we will establish results (well-posedness, existence
and uniqueness of invariant states) for
the systems \eqref{eq:B:eqn:vel:2}--\eqref{eq:theta} and \eqref{eq:vel:inf}--\eqref{eq:theta:inf}, but these
results translate easily back to the original variables in
\eqref{eq:B:eqn:vel}--\eqref{eq:B:eqn:temp} and \eqref{eq:B:vel:inf}--\eqref{eq:B:temp:inf}.

\subsection{Functional Setting}

The equations \eqref{eq:B:eqn:vel:2}--\eqref{eq:theta} supplemented with \eqref{eq:bc:3} may be posed
mathematically as follows.  Define the phase space $H = H_1 \times H_2$ with
\begin{align*}
  H_1 &= \{ \bfU \in ( L^2(\DD))^{d} : \nabla \cdot \bfU = 0, \bfU\cdot \mathbf{n}_{| x_d = 0,1} = 0, \bfU \textrm{ is periodic in $(x_1,\ldots,x_{d-1})$}  \}, \\
  H_2 &= \{  \Trs \in L^2(\DD) : \Trs \textrm{ is periodic in $(x_1,\ldots,x_{d-1})$}  \}.
\end{align*}
Here $\mathbf{n} = (0,\ldots,0, \pm 1)$ is the outward normal to $\DD$.
We next set $V = V_1 \times V_2$ with
\begin{align*}
  V_1 &= \{ \bfU \in ( H^1(\DD))^{d} : \nabla \cdot \bfU = 0, \bfU_{| x_d = 0,1} = 0, \bfU \textrm{ is periodic in $(x_1,\ldots,x_{d-1})$}  \}, \\
  V_2 &= \{  \Trs \in H^1(\DD) : \Trs_{| x_d = 0,1} = 0, \Trs \textrm{ is periodic in $(x_1,\ldots,x_{d-1})$}  \}.
\end{align*}
For further background on this general functional setting see e.g. \cite{ConstantinFoias88,Temam2001}.

We have the following general well-posedness results concerning \eqref{eq:B:eqn:vel:2}--\eqref{eq:bc:3}.  The cases
$d =2$ and $d =3$ are quite different reflecting the situation encountered for the Navier-Stokes equations (deterministic or stochastic) in $d = 2,3$. We
start with $d =2$.

\begin{Prop}
\label{prop:existence-uniqueness:2d}
Suppose $d=2$.  Fix a stochastic basis $\mathcal{S} = (\Omega, \mathcal{F}, \Prb, \{\mathcal{F}_t\}_{t \geq 0}, W)$ and
any $U_0 = (\bfU_0, \Trs_0) \in L^2(\Omega, H)$ which is $\mathcal{F}_0$-measurable relative to this
basis.  Then there exists a unique
\begin{align}
  U = (\bfU, \Trs) \in L^2( \Omega; L^2_{loc}([0,\infty); V) \cap C([0,\infty); H))
  \label{eq:U:reg}
\end{align}
which is predictable (in particular $\mathcal{F}_t$-adapted) satisfying \eqref{eq:B:eqn:vel:2}--\eqref{eq:bc:3}
weakly.
Adopting the notation $U(t, U_0)$ for the solution of \eqref{eq:B:eqn:vel:2}--\eqref{eq:bc:3} corresponding to a given (deterministic) $U_0 \in H$
we have that $U$ is continuous in $U_0$ for every fixed $t \geq 0$.  As such, \eqref{eq:B:eqn:vel:2}--\eqref{eq:bc:3} generates a
Markov semigroup according to
\begin{align*}
  P_t \phi(U_0) = \E \phi(U(t,U_0))
\end{align*}
for any bounded measurable
$\phi : H \to \RR$.  Furthermore, $\{P_t\}_{t \geq 0}$ is Feller, that is,
$P_t$ maps bounded continuous
functions to bounded continuous functions for every $t \geq 0$.
\end{Prop}

In $d = 3$ the results are weaker, a reflection of our incomplete understanding of the 3d Navier-Stokes equation at present.  In what follows
we use $Pr(H)$ to denote the space of Borel probability measures on $H$, and let $\mathcal{B}(H)$ denote the
space of Borel measurable subsets of $H$.
\begin{Prop}
\label{prop:existence-uniqueness:3d}
Suppose $d=3$.  Given any $\mathfrak{m} \in Pr(H)$ with $\int \|(\bfU,\Trs)\|_{L^2}^2 d\mathfrak{m}(\bfU,\Trs) < \infty$, there exists a stochastic
basis $\mathcal{S} = (\Omega, \mathcal{F}, \Prb, \{\mathcal{F}_t\}_{t \geq 0}, W)$ and a stochastic process $U=(\bfU,\Trs)$, relative to that basis,
with $$U \in L^2( \Omega; L^2_{loc}([0,\infty); V)\cap L^\infty_{loc}([0,\infty); H)).$$  Also, $U$ is $\mathcal{F}_t$ adapted, weakly continuous, satisfies
\eqref{eq:B:eqn:vel:2}--\eqref{eq:bc:3}, and the law of $U(0)$ is $\mathfrak{m}$.  Moreover, such a process $U$ exists satisfying the energy inequalities
\eqref{eq:u:ito} and \eqref{eq:theta:ito}, and if we assume that for some $p\geq 3$ we have $\int \|\Trs\|_{L^p}^2 d\mathfrak{m}(\Trs) < \infty$,
then we can further suppose that $\Trs\in L^2(\Omega;L^{\infty}_{loc}([0,\infty);L^p))$.
\end{Prop}

\begin{Rmk}
  Notice that the stochastic elements in the problem are fixed in advance
  in $d=2$ and
  are obtained as a part of the solution in $d =3$.  We say that the solutions given in Proposition~\ref{prop:existence-uniqueness:2d}
  are 'pathwise' solutions whereas those in $d =3$ are `Martingale' solutions.
\end{Rmk}

The existence and uniqueness results given in Propositions~\ref{prop:existence-uniqueness:2d},~\ref{prop:existence-uniqueness:3d} can be established with the
aid of a Faedo-Galerkin scheme.  One establishes sufficient compactness from standard a priori estimates to pass to a limit on a new stochastic
basis using the Skorokhod embedding theorem.   Since a priori estimates yield (pathwise) uniqueness results for the 2 dimensional case,
convergence in the given stochastic basis may be recovered from a variation of the Yamada-Watanabe theorem found in \cite{gyongy1996}.
The details of these proofs are technically involved but follow very closely the analysis in numerous previous works, see e.g.
\cite{ZabczykDaPrato1992, FlandoliGatarek1, Bensoussan1, GlattHoltzZiane2, DebusscheGlattHoltzTemam1} for further details.

We also have the following well-posedness result for the infinite Prandtl system \eqref{eq:vel:inf}--\eqref{eq:theta:inf}.
\begin{Prop}
Suppose $d=2$ or $d=3$.  Fix a stochastic basis $\mathcal{S}$ and any $\mathcal{F}_0$-measurable random variable $\Tho_0 \in L^2(\Omega, H_2)$.  Then there
exists a unique process
\begin{align*}
   \Tho \in L^2(\Omega; L^2_{loc}([0,\infty); V_2) \cap C([0,\infty); H_2)),
\end{align*}
which is $\mathcal{F}_{t}$-adapted, weakly solves \eqref{eq:vel:inf}--\eqref{eq:theta:inf} and satisfies the initial condition $\Tho(0) = \Tho_0$.
Adopting the notation $\Tho(t, \Tho_0)$ for the solution of \eqref{eq:vel:inf}--\eqref{eq:theta:inf} corresponding to a given (deterministic)
$\Tho_0 \in H_2$ we have that $\Tho$ is continuous in $\Tho_0$ for every fixed $t \geq 0$.
As such, \eqref{eq:vel:inf}--\eqref{eq:theta:inf} generates a Feller
Markov semigroup according to
\begin{align*}
  P_t \phi(\Tho_0) = \E \phi(\Tho(t,\Tho_0))
\end{align*}
for any bounded measurable
$\phi : H_2 \to \RR$.
\label{prop:existence-uniqueness:inf}
\end{Prop}

\begin{Rmk}\label{rem:inf}
We remark that Propositions \ref{prop:existence-uniqueness:2d}, \ref{prop:existence-uniqueness:3d},
and \ref{prop:existence-uniqueness:inf} hold with $N_1=\infty$ or $N_2=\infty$ (note that $N_1=0$ for \eqref{eq:vel:inf}--\eqref{eq:theta:inf})
provided we impose sufficient decay in the bases $\{\tilde{\sigma}_{k}\}$ and $\{\sigma_{k}\}$ for $H_1$ and $H_2$, respectively.
For example, if we impose that $\sum_{k=1}^{\infty}\|\tilde{\sigma}_{k}\|_{H^s}^{2}<\infty$ and $\sum_{k=1}^{\infty}\|\sigma_{k}\|_{H^s}^{2}<\infty$
for $s>0$ sufficiently large, then these propositions hold as stated.
\end{Rmk}

\subsection{Notation}

We denote the $L^2$ norm by $\| \cdot\|$ in all that follows.
To simplify notation we assume that the domain $\DD$ is such that the Poincar\' e constant is equal to one:
\begin{align*}
\|\nabla f\| \geq \|f\|.
\end{align*}
We use $C$ to denote various constants which may depend on the size of the domain.  If a constant depends on other quantities, we
will state this explicitly.

\section{A Priori Estimates}
\label{sec:3}

In this section we collect some a priori estimates on solutions to \eqref{eq:B:eqn:vel:2}--\eqref{eq:theta}, \eqref{eq:vel:inf}--\eqref{eq:theta:inf} and related systems.  We begin by establishing
estimates on solutions for a class of stochastic drift-diffusion equations which appear throughout our analysis.
The proofs will rely on a comparison argument, an application of the maximum principle, and certain exponential martingale inequalities.

\subsection{Bounds for a class of Stochastic Drift-Diffusion Equations}
\label{sec:mom:est:drift:diff:eqn}

Fix a stochastic basis $\mathcal{S} = (\Omega, \mathcal{F}, \Prb, \{\mathcal{F}_t\}_{t \geq 0}, W)$ and
consider the following stochastic drift diffusion equation
\begin{align}
  d \xi + \bfV \cdot \nabla \xi dt = (\rab \cdot v_d + \Delta \xi)dt + \sum_{k =1}^N \sigma_k dW^k, \quad \xi(0) = \xi_0
  \label{eq:bous:drift:diff:stoch:sys:1}
\end{align}
evolving on the domain $\DD = [ 0, L]^{d-1} \times [0,1]$,
where $\bfV$ is a is weakly continuous, $\mathcal{F}_t$-adapted, and
divergence free vector satisfying
\begin{equation*}
\bfV = (v_1, \ldots,v_d)
\in L^2(\Omega;L^2_{loc}([0,\infty); V_1) \cap L^{\infty}_{loc}([0,\infty); H_1)) \,.
\end{equation*}
and both $\bfV$ and $\xi$ satisfy the mixed Dirichlet-periodic
boundary conditions
\begin{align}
      \bfV_{|x_d = 0} = \bfV_{|x_d = 1}  = 0,\quad  \xi_{|x_d = 0} =  \xi_{|x_d = 1} = 0, \quad \bfV,
      \xi  \textrm{ are periodic in } (x_1, \ldots,x_{d-1}).
   \label{eq:bc:homo:gen:1}
\end{align}
Recall that by the change of variable $T = \xi + \rab(1 - x_d)$ we may reformulate \eqref{eq:bous:drift:diff:stoch:sys:1} as
\begin{align}
  d T + \bfV \cdot \nabla T dt = \Delta Tdt + \sum_{k =1}^N \sigma_k dW^k, \quad T(0) = T_0  = \xi_0 + \rab(1 - x_d) \,,
    \label{eq:bous:drift:diff:stoch:sys:2}
\end{align}
where
\begin{align}
      \bfV_{|x_d = 0} = \bfV_{|x_d = 1}  = 0,\quad  T_{|x_d = 0} = \rab, \;  T_{|x_d = 1} = 0, \quad \bfV,
      T \textrm{ are periodic in }  (x_1, \ldots,x_{d-1}).
   \label{eq:bc:inhomo:gen}
\end{align}
Let $q \geq 2$ and consider any initial condition $\xi_0 \in H_2 \cap L^q(\DD)$ which is $\mathcal{F}_0$-measurable and
assume
\begin{align*}
   \E \exp( \eta^* \| \xi_0 \|_{L^q}^2) < \infty,
\end{align*}
for some $\eta^* > 0$.
We say $\xi$ is a solution of \eqref{eq:bous:drift:diff:stoch:sys:1} 
if it is weakly continuous, $\mathcal{F}_t$-adapted, and satisfies
\begin{align}
\label{eq:xi:reg}
\xi\in L^2(\Omega;L^2_{loc}([0,\infty); V_2) \cap L^{\infty}_{loc}([0,\infty);L^q)
\cap C([0,\infty); H_2)) \,,
\end{align}
where for $d = 3$ we further assume 
$q \geq 3$ or $q \geq 2$ and
$\bfV \in L^2(\Omega;L^2_{loc}([0,\infty); H^2(\mathcal{D}))$. 
Furthermore, the corresponding $T$ satisfies \eqref{eq:bous:drift:diff:stoch:sys:2} in weak sense (recall $\bfV$ is divergence free). 
Under these assumptions, there exists a unique solution of \eqref{eq:bous:drift:diff:stoch:sys:1} (see e.g. \cite{ZabczykDaPrato1992}).

\begin{Prop}
\label{prop:dde}
For any $2\leq p\leq q$, there exist constants $C_0=C_0(p)$ (with $C_0(2)=1$), $C=C(\mathcal{D},p)$ and $\gamma = \frac{C_0}{\|\sigma\|_{L^p}^2}$
 such that for each $K>0$, the solution $\xi$ of \eqref{eq:bous:drift:diff:stoch:sys:1} satisfies
\begin{multline}\label{eq:lppe-1}
\Prb \left(\left. \sup_{s\geq 0} \left( \frac{1}{2} \|\xi(s)\|_{L^p}^2
+ \frac{C_0(p)}{2} \int_0^s \|\xi\|_{L^p}^2 ds'
 -  2\|\xi_0\|_{L^p}^2  \right.\right.\right. \\
 \left.\left.\left.\vphantom{\int} -  ((p-1) \|\sigma\|_{L^p}^2 +
 C\rab^2) s \right) - C\rab^2 \geq K \, \right| \mathcal{F}_0 \right)
\leq e^{-\gamma K}\,.
\end{multline}
Furthermore, there exists $\eta_0 = \eta_0(\rab, p,\|\sigma\|_{L^p})$ such that for any $t \geq 0$, and any $\eta \leq \eta_0\wedge (\eta^*/4)$,
\begin{align}
    \E \exp \left( \eta \sup_{s \in [0, t]} \|\xi(s)\|_{L^p}^2 \right) \leq C' \E \exp \left( 4\eta  \|\xi_0\|_{L^p}^{2}\right) \,,
    \label{eq:gen:drift:diff:bnd:1}
\end{align}
for a constant $C' = C'( \rab,p,\|\sigma\|_{L^p},|\DD|,t)$ independent of $\eta$ and $\xi_0$.
Also, for any fixed $t>0$, letting $\psi (s) := \exp\left(C_0 (s - t)/p \right)$, we have
for each $K>0$,
\begin{multline} \label{eq:lppe-2}
\Prb \left(\left. \sup_{s \in [0, t]} \left( \frac{1}{2}\psi(s) \|\xi(s)\|_{L^p}^2
+ \frac{C_0}{2p} \int_0^s \psi(s') \|\xi\|_{L^p}^2 ds'
 -
2\psi(0) \|\xi(0)\|_{L^p}^2 \right.\right.\right. \\
\left. \left. \left.
\vphantom{\int} - \frac{p(p-1)}{C_0}  \|\sigma\|_{L^p}^2
- C \rab^2  \right) \geq K \, \right| \mathcal{F}_0 \right) \leq e^{-\gamma K} \,,
\end{multline}
and for
$\eta \leq \eta_0\wedge (\eta^*/4)$,
\begin{align}
  \E \exp \left( \eta \|\xi(t)\|^2_{L^p}  \right)
  \leq C''
  \E \exp \left( 4\eta (e^{- \kappa t} \|\xi_0\|^2)   \right) \,,
      \label{eq:gen:drift:diff:bnd:2}
\end{align}
where $C'' = C''( \rab,p, \|\sigma\|_{L^p}, |\DD|)$ and $\kappa = \kappa(\rab,p, |\DD|) > 0$ are independent of $t$, $\eta$ and $\xi_0$.
\label{prop:lets:make:the:most:of:the:moments:that:count}
\end{Prop}

\begin{proof}
In order to establish the desired bounds we take advantage of a comparison principle to work with a homogeneous version of
\eqref{eq:bous:drift:diff:stoch:sys:2}--\eqref{eq:bc:inhomo:gen} in place of \eqref{eq:bous:drift:diff:stoch:sys:1}--\eqref{eq:bc:homo:gen:1}.
Consider $S$, the solution to
\begin{align}
& dS + \bfV \cdot \nabla S dt = \Delta S dt + \sum \sigma_{k} d W^{k}, \quad S(0)= T_{0}, \notag \\
&S_{|x_d=0}= S_{|x_d=1} = 0, \quad S \ \text{periodic in}\ (x_1, \ldots,x_{d-1}).
\label{eq:S:def}
\end{align}
Note that, under the assumed conditions on the regularity of $\sigma$ and $\bfV$, the systems \eqref{eq:bous:drift:diff:stoch:sys:2} and \eqref{eq:S:def}
 possess unique pathwise solutions.
Then since $\tilde{T} = T - \rab$ solves
\begin{align*}
&d \tilde{T} + \bfV \cdot\nabla \tilde{T} dt = \Delta \tilde{T} dt + \sum\sigma_{k}dW^{k}, \quad \tilde{T}(0)=T_{0}-\rab, \\
&\tilde{T}_{|x_d=0}= 0, \; \tilde{T}_{|x_d=1} = -\rab, \quad \tilde{T} \ \text{periodic in}\ (x_1, \ldots,x_{d-1}),
\end{align*}
we have that $R=S-\tilde{T}$ is the unique solution of
\begin{align}
\label{eq:R}
&\pd{t} R + \bfV \cdot \nabla R = \Delta R, \quad R(0)= \rab > 0, \\
& R_{|x_d=0}= 0, \; R_{|x_d=1} = \rab, \quad R \ \text{periodic in}\ (x_1, \ldots,x_{d-1}).
\end{align}
Hence, by the maximum principle, we infer that $R\geq 0$ in $\DD\times[0,\infty)$.

Since $R\geq 0$ in $\DD\times[0,\infty)$, this gives $S\geq T - \rab$, so that
\begin{align}
\label{eq:Tbound1}
(T- \rab)_{+} \leq S_{+}.
\end{align}
On the other hand, $\tilde{R}=S - T$ solves
\begin{align*}
&\pd{t} \tilde{R} + \bfV \cdot \nabla \tilde{R} = \Delta \tilde{R}, \quad \tilde{R}(0)=0, \\
& \tilde{R}_{|z=0}= - \rab, \quad \tilde{R}_{|z=1} = 0, \quad \tilde{R} \ \text{periodic in}\ \mathbf{x} =  (x_1, x_2).
\end{align*}
Hence, again due to the maximum principle,
$\tilde{R}\leq 0$ in $\DD\times[0,\infty)$, which gives $T\geq S$ and we have
\begin{align}
\label{eq:Tbound2}
T_{-}\leq S_{-}.
\end{align}
Combining \eqref{eq:Tbound1} and \eqref{eq:Tbound2} yields
$|T|= T_{+}+T_{-}\leq (T- \rab)_{+}+T_{-}+ \rab \leq |S| + \rab$ so that
\begin{align}
|\xi |\leq |T|+\rab \leq |S|+2 \rab.
\label{eq:theta:S:bound}
\end{align}

With the bound \eqref{eq:theta:S:bound} in mind we proceed to estimate the solution $S$ of
\eqref{eq:S:def} as follows.  By the $L^p$ It\=o lemma (see \cite{Krylov2010}),
\begin{equation}
 d \| S\|_{L^p}^p - p \int_{\DD} \Delta S \cdot S |S|^{p-2} dx dt = \frac{p(p - 1)}{2} \sum_{k =1}^N \int_{\DD} \sigma_k^2 |S|^{p - 2} dx dt +
p \sum_{k =1}^N \int_{\DD} \sigma_k S |S|^{p - 2} dx dW_k \,,
\end{equation}
where we used that $\bfV$ is divergence free to drop the advective terms $\bfV \cdot \nabla S$. Since we seek to estimate exponential moments for
$\|S\|_{L^p}^2$ we next make a second application of It\={o}'s lemma with $\phi_p(x) := (\delta + x)^{2/p}$ for any $\delta > 0$, $p \geq 2$.
Direct computation yields
\begin{align*}
   \phi_p'(x) =  \frac{2}{p} (\delta + x)^{\frac{2-p}{p}}, \quad    \phi_p''(x) =  \frac{2(2-p)}{p^2} (\delta + x)^{\frac{2-2p}{p}}.
\end{align*}
Noting that $\phi''(x)$ is non-positive for every $x \geq 0$, we can drop the It\=o
correction term and infer
\begin{align}
\frac{1}{p} d \phi_p( \| S\|_{L^p}^p ) -
  \phi_p'(\| S\|_{L^p}^p) \int_{\DD} \Delta S \cdot S |S|^{p-2} dx dt &\leq
        \frac{(p - 1)}{2} \phi_p' (\| S\|_{L^p}^p) \sum_{k =1}^N \int_{\DD} \sigma_k^2 |S|^{p - 2} dx dt \notag \\
      &\quad +  \phi_p'(\| S\|_{L^p}^p)
       \sum_{k =1}^N \int_{\DD} \sigma_k S |S|^{p - 2} dx dW_k \,.
\label{eq:lp:ito}
\end{align}

We can integrate by parts (due to the zero Dirichlet boundary
condition in \eqref{eq:S:def}) and use a version of the Poincar\' e inequality (see \cite[Proposition 7.14.1]{KuksinShirikian12}) to obtain
\begin{align}
- \int_{\DD} \Delta S S |S|^{p-2} dx =  (p - 1) \int_{\DD} |S|^{p - 2} |\nabla S|^2 \, dx \geq
C_0 \int_{\DD} |S|^p \, dx = C_0\|S\|^p_{L^p}\,,
\label{eq:lp:1}
\end{align}
where $C_0 = C_0(p)$ with $C_0(2) = 1$.
On the other hand
\begin{align}
    \phi_p'(\| S\|_{L^p}^p) \sum_{k =1}^N \int_{\DD} \sigma_k^2 S^{p - 2} dx
  \leq    \phi_p'(\| S\|_{L^p}^p) \|\sigma\|_{L^p}^2 \|S\|^{p-2}_{L^p}
  \leq \frac{2}{p}  \|\sigma\|_{L^p}^2.
  \label{eq:Lp:ito:cor:est}
\end{align}
Finally note that the final term in \eqref{eq:lp:ito} is the differential of a (local) Martingale whose quadratic variation
is given by
\begin{align*}
   \int_0^t (\phi_p'(\| S\|_{L^p}^p))^2 \sum_{k =1}^N \left(\int_{\DD} \sigma_k S |S|^{p - 2} dx\right)^2 ds \,.
\end{align*}
We bound this term using H\" older's inequality
\begin{align}
  \int_0^t (\phi_p'(\| S\|_{L^p}^p))^2 \sum_{k =1}^N \left(\int_{\DD} \sigma_k S |S|^{p - 2} dx\right)^2 ds
  &\leq \int_0^t (\phi_p'(\| S\|_{L^p}^p))^2 \left(\|\sigma\|_{L^p}
\|S\|_{L^p}^{p-1}
  \right)^2 ds \notag \\
  &=  \frac{2}{p} \|\sigma\|_{L^p}^2
  \int_0^t \phi_p'(\| S\|_{L^p}^p)(\delta + \| S\|_{L^p}^p)^{\frac{2 - p}{p}}  \| S\|_{L^p}^{2(p-1)} ds \notag \\
  &\leq  \frac{2}{p}\|\sigma\|_{L^p}^2\int_0^t \phi_p'(\| S\|_{L^p}^p)\| S\|_{L^p}^p ds.
  \label{eq:Lp:quad:var:term}
\end{align}

We will now use the following estimate: for any continuous local martingale $M_t$ with $M_{0}=0$, for any $\gamma,K>0$,
\begin{align}
\label{eq:mart}
\Prb\left(\sup_{t\geq 0}\left(M_t -\frac{\gamma}{2}\langle M\rangle_{t} \right)\geq K\right) \leq e^{-\gamma K}.
\end{align}
In particular, we can combine the bounds \eqref{eq:lp:ito}--\eqref{eq:Lp:quad:var:term} with the exponential martingale estimate \eqref{eq:mart} for any
$\gamma \leq \frac{C_0}{\|\sigma\|_{L^p}^2}$ to infer, by passing to the limit as $\delta \rightarrow 0$,
\begin{align}
\label{eq:1}
\Prb \left(\left. \sup_{s \geq 0} \left(\|S(s)\|_{L^p}^2 +
C_0(p) \int_0^s \|S\|_{L^p}^2 ds'  -  \|S_0\|_{L^p}^2 - (p-1) s \|\sigma\|_{L^p}^2 \right) \geq K \, \right| \mathcal{F}_0 \right) \leq e^{-\gamma K} \,.
\end{align}
Note that from \eqref{eq:theta:S:bound} and Minkowski inequality, we have
\begin{align}
\label{eq:2}
\|\xi\|_{L^p}^2 \leq (\|S\|_{L^p} + 2\rab |\DD|^{1/p} )^2 \leq
2\|S\|_{L^p}^2 + C \rab^2
\end{align}
and
\begin{align}
\label{eq:3}
\|S_0\|_{L^p}^2 = \|T_0\|_{L^p}^2 = \|\xi_0 + \rab(1 - z)\|_{L^p}^2
\leq 2 \|\xi_0\|_{L^p}^2 + C \rab^2 \,,
\end{align}
where $C = C(\DD)$. Then \eqref{eq:lppe-1} follows by combining \eqref{eq:1}--\eqref{eq:3}.
Also the bound \eqref{eq:gen:drift:diff:bnd:1} is obtained from \eqref{eq:lppe-1} by invoking the elementary formula
\begin{align}
\label{eq:cake}
\E |X| = \int_{0}^\infty \Prb(|X| \geq y) \, dy.
\end{align}

In order to establish the remaining bounds, \eqref{eq:lppe-2}--\eqref{eq:gen:drift:diff:bnd:2}, we fix
$t > 0$ and set $\psi (s) := \exp\left(C_1 (s - t) \right)$ with
$C_1 = \frac{C_0}{p}$. Using It\=o formula for
$\psi(s) \phi_p( \|S(s)\|^p_{L^p})$
yields (cf. \eqref{eq:lp:ito})
\begin{multline}
\frac{1}{p} d (\psi \phi_p( \| S\|_{L^p}^p )) -
 \psi \phi_p'(\| S\|_{L^p}^p) \int_{\DD} \Delta S \cdot S |S|^{p-2} dx dt \leq
        \frac{(p - 1)}{2} \psi \phi_p'(\| S\|_{L^p}^p) \sum_{k =1}^N \int_{\DD} \sigma_k^2 |S|^{p - 2} dx dt \notag \\
      \quad + \frac{C_1}{p} \psi \phi_p(\| S\|_{L^p}^p) dt +
       \psi \phi_p'(\| S\|_{L^p}^p) \sum_{k =1}^N \int_{\DD} \sigma_k S |S|^{p - 2} dx dW_k \,.
\label{eq:lp:ito:2}
\end{multline}
Using $\psi \leq 1$, $\int_0^t \psi \leq 1$, and
the exponential martingale inequality \eqref{eq:mart}, we obtain as above
\begin{multline}
\Prb \left( \sup_{s \in [0, t]} \left( \psi(s) \phi_p(\|S(s)\|_{L^p}^p)
+  C_0 \int_0^s \psi(s') \Big( \phi'_p(\|S\|_{L^p}^p) \|S\|_{L^p}^p -
\frac{1}{p} \phi_p(\|S\|_{L^p}^p) \Big)
 ds'
\right. \right.
\\
\left. \left. \left.
\vphantom{\int}
-\psi(0) \phi_p( \|S(0)\|_{L^p}^p) - \frac{p(p-1)}{C_0} \|\sigma\|_{L^p}^2 \right) \geq K \, \right| \mathcal{F}_0 \right) \leq e^{-\gamma K} \,.
\end{multline}
By observing that $x\phi'_p(x) - \frac{1}{p}\phi(x) \to
\frac{1}{p} x^{\frac{2}{p}}$ as $\delta \to 0$, and changing $S$ to $\xi$, we obtain \eqref{eq:lppe-2} from the dominated convergence theorem.
Choosing $s = t$ and using $\psi(t) = 1$ and
\eqref{eq:cake} we obtain
\eqref{eq:gen:drift:diff:bnd:2}, completing the proof.
\end{proof}

\subsection{Bounds for the Boussinesq System}

In this section we establish bounds on solutions to
\eqref{eq:B:eqn:vel:2}--\eqref{eq:theta}.  These proofs will invoke Proposition \ref{prop:lets:make:the:most:of:the:moments:that:count}
and the exponential martingale inequality \eqref{eq:mart},
and will rely on carefully chosen weighted norms to deal with interacting terms in the system.

\begin{Prop} \label{lem:exp-bound}
Fix a stochastic basis $\mathcal{S} = (\Omega, \mathcal{F}, \Prb, \{\mathcal{F}_t\}_{t \geq 0}, W)$,
and let $U_0 = (\bfU_0, \Trs_0)\in H$ be $\mathcal{F}_0$-measurable initial conditions.
\begin{enumerate}[(i)]
\item When $d=2$, let $U = (\bfU, \Trs)$ denote the corresponding pathwise solution of \eqref{eq:B:eqn:vel:2}--\eqref{eq:theta}.
\item When $d=3$, assume $\E\exp(\eta^* \|\Trs_0\|_{L^3}^{2})<\infty$ for some $\eta^*>0$, and let $U= (\bfU, \Trs)$ denote
a corresponding martingale solution of \eqref{eq:B:eqn:vel:2}--\eqref{eq:theta} satisfying the conclusions
of Proposition \ref{prop:lets:make:the:most:of:the:moments:that:count} (with $q=3$).
\end{enumerate}
Let
\begin{align}
\gamma &:= \left\{\begin{array}{ll}
\min\left(\frac{1}{4\prN\|\tilde{\sigma}\|^2},\frac{1}{4}\right), \hspace{0.2in} \text{if}\
\|\tilde{\sigma}\|>0,\\
\frac{1}{4} \hspace{1.23in}
\text{if}\  \|\tilde{\sigma}\|=0.
\end{array}\right.
\end{align}
Then there exist constants $C_2, C_3$, depending on $\ra,\rab,\prN$, and $\|\tilde{\sigma}\|$,
and a random constant $C_1$ which further depends on $\|\bfU_0\|,\|\Trs_0\|,$ such that for each $K>0$,
\begin{multline}\label{eq:prb:est}
	\Prb  \left(\sup_{t\geq 0}\left(\|\bfU(t)\|^{2}
			 +  \frac{\prN}{2}\|\Trs(t)\|^2
                         +\frac{\prN}{2}  \int_{0}^{t}  \left(\|\nabla \bfU(s)\|^2 + \frac{1}{2}\|\nabla\Trs(s)\|^2  \right) ds
                         \right.\right. \\
                         \left.\left.\left.\vphantom{\int}- C_2t - C_1 \right)\geq C_3 K
                         \right|\mathcal{F}_0\right) \leq 3 e^{-\gamma K}.
\end{multline}
Moreover, if $\prN\geq 2$, then the constants $\eta_1,C_4,\kappa>0$ given by
\begin{align}
\eta_1 :&= \frac{1}{64\ra^2 + 8} \left\{\begin{array}{ll}\min\left(\frac{1}{2\prN\|\tilde{\sigma}\|^2},1\right), \hspace{0.2in} \text{if}\
\|\tilde{\sigma}\|>0,\\
1 \hspace{1.22in} \text{if} \ \|\tilde{\sigma}\|=0,
\end{array}\right. \\
  C_4 :&= 4\prN \|\tilde{\sigma}\|^2 + C(1 +  \ra^2 )(1+ \rab^2), \;\;  \kappa :=  C(1 + \ra^2),
  \label{eq:const:exp:decay:mom}
\end{align}
are such that for any $0<\eta < \eta_1$ and $t^* > 0$, we have
\begin{align}\label{eq:exp:est:2}
&\E  \exp \left(\eta  \|U(t^*)\|^{2}\right)
   \leq C\E \exp\left( \eta ( \kappa e^{ - t^*/2}\|U_0\|^2  + C_4) \right).
\end{align}
\end{Prop}
\begin{Rmk}
We emphasize that when $\tilde{\sigma} = 0$ (i.e. no stochastic forcing in the velocity component) the constants $\gamma$, $\eta_1$, $C_4$ and $\kappa$
 are \textit{independent} of $\prN$.
\end{Rmk}
\begin{proof}
Using \eqref{eq:B:eqn:vel:2} we compute with the  It\={o} formula (note that this is an inequality for $d=3$)
\begin{align}
\label{eq:u:ito}
d\|\bfU\|^{2} + 2\prN\|\nabla \bfU\|^{2}dt
\leq \left(2\prN\ra\langle \Trs,u_d\rangle  + (\prN)^{2}\|\tilde{\sigma}\|^{2}\right)dt + 2\prN\sum\langle \tilde{\sigma}_k,\bfU\rangle d\tilde{W}^{k}.
\end{align}
The last term in the previous line is a Martingale with quadratic variation
$4(\prN)^{2}\int_{0}^{t}\sum\langle \tilde{\sigma}_{k},\bfU\rangle^{2}ds$.  Then since
\begin{align*}
\sum\langle \tilde{\sigma}_{k},\bfU\rangle^{2} \leq \|\tilde{\sigma}\|^{2}\|\bfU\|^{2},
\end{align*}
by using the Poincar\'{e} inequality, one has
\begin{align*}
&d\|\bfU\|^{2} + \left(\prN\|\nabla \bfU\|^{2}-\prN(\ra)^2\|\Trs\|^2
-(\prN)^{2}\|\tilde{\sigma}\|^{2}-2\gamma(\prN)^{2}
\|\tilde{\sigma}\|^{2}\|\bfU\|^{2}\right)dt \\
&\quad \quad \quad \quad \leq  -2\gamma(\prN)^{2}\sum\langle \tilde{\sigma}_{k},\bfU\rangle^{2}dt
+ 2\prN\sum\langle \tilde{\sigma}_k,\bfU\rangle d\tilde{W}^{k}.
\end{align*}
Applying \eqref{eq:mart} with $\gamma \leq 1/(4\prN\|\tilde{\sigma}\|^2)$, we have for any $K>0$, that $\Prb (E_{1, K}) \leq e^{-\gamma K}$ with
\begin{align}\label{eq:E1}
\begin{split}
&E_{1, K} := \left\{\sup_{t\geq 0}\left(\|\bfU(t)\|^{2}
+\int_{0}^{t}\left(\frac{\prN}{2}\|\nabla \bfU(s)\|^2-\prN(\ra)^2\|\Trs(s)\|^{2}\right)ds  \right.\right.
\\
& \hspace{3in} \left.\left.\vphantom{\int_0^t}
-(\prN)^{2}\|\tilde{\sigma}\|^{2}t -\|\bfU_0\|^{2}\right)\geq K\right\} .
\end{split}
\end{align}
Note that when $\|\tilde{\sigma}\|=0$, this holds for any $\gamma>0$.
Next,
since $\Trs$ satisfies a drift diffusion equation of the form \eqref{eq:bous:drift:diff:stoch:sys:1},
we obtain from \eqref{eq:lppe-1} that for any $\gamma \leq 1/4$ one has
$\Prb (E_{2, K}) \leq e^{-\gamma K}$, with
\begin{align}\label{eq:E2}
E_{2, K} := \left\{\sup_{t\geq 0}\left(
\|\Trs(t)\|^{2}+\int_{0}^{t}\|\Trs(s)\|^2 ds - (C\rab^2 +  2)t - (C\rab^2+4\|\Trs_0\|^2)
\right)
\geq K\right\}  \,,
\end{align}
where we used $C_0(2) = 1$ and $\|\sigma\| = 1$.
Note that from \eqref{eq:theta} one also has
\begin{align}\label{eq:theta:ito}
d\|\Trs\|^{2} + 2 \|\nabla\Trs\|^{2}dt \leq (2 \rab\langle \Trs,u_d\rangle  + 1)dt +
2 \sum\langle \sigma_k,\Trs\rangle dW^{k},
\end{align}
and, with $\gamma \leq 1/4$, another application of \eqref{eq:mart} gives
$\Prb (E_{3, K}) \leq e^{-\gamma K}$, where
\begin{align}
\label{eq:E3}
E_{3, K} :=
\left\{\sup_{t\geq 0}\left(
\|\Trs(t)\|^2 + \int_{0}^{t} \left(\frac{1}{2}\|\nabla\Trs(s)\|^2
- \rab^{2}\|u(s)\|^2\right) ds -  t -\|\Trs_0\|^2
\right)
\geq K\right\}.
\end{align}
Define $E_{K} := \cup_{j=1}^{3} E_{j,K}$.  Then we have
$\Prb(E_{K} ) \leq 3e^{-\gamma K}$.
Furthermore, on the complement of $E_{K}$,
we evaluate $(1+\rab^2)E_{1, K} +  (1+\rab^2)\prN \ra^2 E_{2, K} +  \frac{\prN}{2}E_{3, K}$,
meaning that we are adding the inequalities defining $E_{i, K}$, and using
$\|\nabla u\| \leq \|u\|$ we obtain
for any $t\geq 0$,
\begin{align*}
 &\|\bfU(t)\|^{2} + \frac{\prN}{2}\|\Trs(t)\|^2
 +\frac{\prN}{2}\int_{0}^{t}\left(\|\nabla \bfU(s)\|^2 + \frac{1}{2}\|\nabla \Trs(s)\|^2\right) ds  \\
& \leq
 \left((1+\rab^2)\left(\prN^2\|\tilde{\sigma}\|^{2} + \ra^2\prN(C\rab^2+2)\right) + \frac{\prN}{2} \right)t
\\ & \quad \quad + (1+\rab^2)\left(\|\bfU_0\|^{2} + \ra^2\prN(C\rab^2+4\|\Trs_0\|^{2})\right) + \frac{\prN}{2}\|\Trs_0\|^{2}
 + \left((1+\rab^2)( \prN\ra^{2}+1) + \frac{\prN}{2}\right)K.
\end{align*}
By defining $C_1,C_2$ and $C_3$ as
\begin{align}
C_1 &:= (1+\rab^2)\left(\|\bfU_0\|^{2} + \ra^2\prN(C\rab^2+4\|\Trs_0\|^{2})\right) + \frac{\prN}{2}\|\Trs_0\|^{2}\,,
\label{eq:def-c1}\\
C_2 &:= (1+\rab^2)\left(\prN^2\|\tilde{\sigma}\|^{2} + \ra^2\prN(C\rab^2+2)\right) + \frac{\prN}{2}  \,,
\label{eq:def-c2}\\
C_3 &:= (1+\rab^2)( \prN\ra^{2}+1) + \frac{\prN}{2}
\label{eq:def-c3} \,,
\end{align}
the proof of \eqref{eq:prb:est} is complete.

It remains to prove \eqref{eq:exp:est:2}.  Fix $t^*>0$ and define $\phi(t) := \exp(\frac{\prN}{2} (t - t^*))$. It follows from
\eqref{eq:u:ito} that
\begin{align*}
   d (\phi \| \bfU \|^2) + 2\phi \prN \| \nabla \bfU\|^2 dt
    =  \phi \left(2\prN\ra\langle \Trs,u_2\rangle  + \frac{\prN}{2} \|\bfU\|^2 +
     (\prN)^{2}\|\tilde{\sigma}\|^{2}\right)dt + 2\prN \phi \sum\langle \tilde{\sigma}_k,\bfU\rangle d\tilde{W}^{k}.
\end{align*}
Rearranging and applying standard estimates as above we find
\begin{align*}
   d (\phi \| \bfU \|^2) +& \phi \left(\frac{\prN}{2} \| \nabla \bfU\|^2 - 2\prN\ra^2\|\Trs\|^2
   - 2\delta \prN^2 \phi \|\tilde{\sigma}\|^{2}\|\bfU\|^{2} - (\prN)^2\|\tilde{\sigma}\|^{2} \right)dt \\
    &\quad \qquad  \leq 2\prN \phi \sum\langle \tilde{\sigma}_k,\bfU\rangle d\tilde{W}^{k}
    - 2\delta \prN^2 \phi^2 \sum\langle \tilde{\sigma}_k,\bfU\rangle^2 dt.
\end{align*}
Set $\delta := \frac{1}{4 \prN \| \tilde{\sigma}\|^2}$
if $\|\tilde{\sigma}\| > 0$, or $\delta := 1$ if $\|\tilde{\sigma}\| = 0$.
Using that $\phi \leq 1$,
and  $\int_0^{t^*} \phi dt < \frac{2}{\prN}$ we obtain that
$\Prb (\tilde{E}_{1,K}) \leq e^{- \delta K} $, where
\begin{align*}
  \tilde{E}_{1,K} := \left\{  \|\bfU(t^*)\|^2 -
  2 \prN \ra^2 \int_0^{t^*} \! \! \phi(t)   \|\Trs\|^2  dt
  - 2\prN \| \tilde{\sigma} \|^2 - e^{- \prN t^*/2} \|\bfU_0\|^2 \geq K \right\}\,.
\end{align*}
Since $\Trs$ satisfies a drift-diffusion equation of the form \eqref{eq:bous:drift:diff:stoch:sys:1}, we obtain from \eqref{eq:lppe-2}
that $\Prb (\tilde{E}_{2,K}) \leq e^{- K/2}$, where $\psi(t) = \exp((t - t^*)/2)$
and
\begin{align*}
  \tilde{E}_{2,K} &:=
\left\{ \sup_{t \in [0,t^*]} \left( \frac{\psi(t)}{2} \|\Trs(t)\|^2 + \frac{1}{2} \int_0^t    \psi \|\Trs\|^2 ds - 2e^{- t^*/2} \|\Trs_0\|^2
- 1 - C \rab^2
\right) \geq K \right\} \\
 &\supseteq  \left\{ \frac{1}{2} \|\Trs(t^*)\|^2 +
 \int_0^{t^*}   \frac{\psi(s)}{2} \|\Trs\|^2 ds
   - 2e^{- t^*/2} \|\Trs_0\|^2 - 1 - C\rab^2 \geq K \right\}\,.
\end{align*}
We now define $\tilde{E}_K = \tilde{E}_{1,K} \cup \tilde{E}_{2,K}$ and infer that, as above
$\Prb(\tilde{E}_K) \leq 2 e^{- \gamma K}$ where $\gamma = \delta \wedge \frac{1}{2}$ is independent of $\prN$ if $\tilde{\sigma} = 0$.
Now on $\tilde{E}_K^c$ we estimate
\begin{align*}
   \|\bfU(t^*)\|^2& + \frac{1}{2} \|\Trs(t^*)\|^2 \\
   &\leq 2 \prN \ra^2 \int_0^{t^*} \phi(t) \|\Trs\|^2 dt + 2\prN \|\tilde{\sigma}\|^2 + e^{- \prN t^*/2 } \|\bfU_0\|^2 + 2 e^{-t^*/2} \|\Trs_0\|^2
   + 2 + C\rab^2 + 2K \,.
\end{align*}
Observing that $\phi^{1/2} \leq \psi$ for $\prN \geq 2$ we obtain on $\tilde{E}_K^c$,
\begin{align*}
\int_0^{t^*}  \phi \|\Trs\|^2 dt &\leq
\int_0^{t^*} \phi^{1/2} \psi \|\Trs\|^2 dt
\leq  \Big( 4e^{- t^*/2} \|\Trs_0\|^2 + 2 + C\rab^2 + 2K \Big)
\int_0^{t^*} \phi^{1/2} dt \\
&\leq \frac{1}{\prN} \left( C\left(e^{- t^*/2} \|\Trs_0\|^2 + 1 + \rab^2\right) + 8K\right).
\end{align*}
Noting the definitions of $C_4,\kappa$ above, we conclude that, for any $t^* > 0$,
\begin{align*}
  \Prb \left( \|U(t^*)\|^2 - C_4 - \kappa e^{- t^*/2} \| U_0\|^2  > (32\ra^2 + 4)K \right) \leq 2 e^{- \gamma K}.
\end{align*}
By using the formula \eqref{eq:cake} we can infer \eqref{eq:exp:est:2} for any $\eta \leq  \gamma/(32\ra^2+4) = \eta_1$, and the proof is complete.
\end{proof}

\subsection{Bounds for the Infinite Prandtl System}

Next we establish estimates analogous to those in Proposition~\ref{lem:exp-bound}
for the infinite Prandlt system \eqref{eq:vel:inf}--\eqref{eq:theta:inf}.

\begin{Prop}
There exist constants $\tilde{C}_1, \tilde{C}_2, \tilde{C}_3 > 0$ depending on $\ra, \rab$ and $\|\Tho_0\|$
such that for each $\gamma \leq 1/4$, and any $K>0$,
the solution $\Tho$ to \eqref{eq:vel:inf}--\eqref{eq:theta:inf} with initial data $\Tho_0$ satisfies
\begin{align}
\label{eq:bd:theta}
\Prb\left(\sup_{t\geq 0}\left(\|\Tho(t)\|^2
+\frac{1}{2}\int_{0}^{t}\|\nabla \Tho(s)\|^2 ds -  \tilde{C}_1-\tilde{C}_2t \right)\geq \tilde{C}_3K \right) &\leq 2 e^{-\gamma K}.
\end{align}
Furthermore, there exists $C,\tilde{C}>0$, $\eta_0=\eta_0(\ra,\rab)>0$ and $\tilde{C}_4=\tilde{C}_4(\ra,\rab)>0$ such that for any $t^*>0$ and $0<\eta<\eta_0$,
\begin{align}
\E\left(\exp\left(\eta\|\Tho\|^{2} + \frac{\eta e^{-t^{*}/4}}{2}\int_{0}^{t^{*}}\|\nabla\Tho\|^{2}ds\right)\right)
\leq C \exp\left(\tilde{C}\eta (1+\ra\rab)e^{-t^{*}/4}\|\Tho_{0}\|^{2}+\tilde{C}_4\right).
\label{eq:exp:theta}
\end{align}
\label{lem:exp-bound-pr}
\end{Prop}

\begin{proof}
The proof of \eqref{eq:bd:theta} is a modification of the proof of Proposition \ref{lem:exp-bound}.  First observe that from \eqref{eq:vel:inf}, the Poincar\'{e} inequality yields
\begin{align}
\|\nabla \Uo\|^2 = \ra\langle u_d^{0},\Tho\rangle \leq \frac{1}{2}\|\nabla\Uo\|^2 + \frac{\ra^{2}}{2}\|\Tho\|^2,
\label{eq:inf:u:id}
\end{align}
and therefore
\begin{align}
\|\Uo\| \leq \|\nabla \Uo\| \leq \ra\|\Tho\|.
\label{eq:inf:u:bound}
\end{align}
Next, we consider the sets $E_{2,K}$ and $E_{3,K}$ as defined in \eqref{eq:E2} and \eqref{eq:E3}, respectively.
Define $A_K = \cup_{j=2,3} E_{j,K}$, and note that by the reasoning from the proof of Proposition \ref{lem:exp-bound} above, for $\gamma \leq 1/4$ we have $\Prb(A_K)\leq 2 e^{-\gamma K}$.  Furthermore, on the complement of $A_K$, for any $t\geq 0$, we can use \eqref{eq:inf:u:bound} to find
\begin{align*}
 \|\Tho(t)\|^2
 +&\frac{1}{2}\int_{0}^{t}\|\nabla \Tho(s)\|^2 ds
 \leq  \|\Tho_0\|^{2} + \rab^{2}\int_{0}^{t}\|\Uo(s)\|^{2}ds + t + K
\\
& \leq \|\Tho_0\|^{2} + \rab^{2}\ra^2\int_{0}^{t}\|\Tho(s)\|^{2}ds  + t + K \\
& \leq  C(\rab^2 \ra^2 + 1)\left(\|\Tho_0\|^{2} +\rab^{2} + \left(\rab^2 + 1\right)t + K\right).
\end{align*}
By defining $\tilde{C}_1, \tilde{C}_2$ and $\tilde{C}_3$ in the obvious way,
we obtain \eqref{eq:bd:theta}.

It remains to establish \eqref{eq:exp:theta}.
Define $\psi(t) = \exp((t - t^{*})/4)$, and note that from \eqref{eq:theta} (and $\|\sigma\|=1$) one has
\begin{align}
\label{eq:theta:ito:2}
d(\psi\|\Tho\|^{2}) +  &\psi\left(\|\nabla\Tho\|^{2}-2\rab \langle \Tho,u^{0}_d\rangle
- 2\psi\gamma\|\Tho\|^{2} - 1\right)dt  \notag \\ &\leq
2 \psi\sum\langle \tilde{ \sigma}_k,\Tho\rangle d\tilde{W}^{k} - 2\psi^2\gamma\sum\langle \tilde{ \sigma}_k,\Tho\rangle^{2}dt.
\end{align}
Observe that by combining \eqref{eq:inf:u:id} and \eqref{eq:inf:u:bound} we have
\begin{align}
\label{eq:inf:1}
\langle \Tho,u^{0}_d\rangle = \frac{1}{\ra}\|\nabla \Uo\|^{2} \leq \ra\|\Tho\|^2.
\end{align}
Then, for $\gamma = 1/4$, the exponential
martingale inequality \eqref{eq:mart} and \eqref{eq:inf:1} give that
$\Prb(F_{1,K}) \leq e^{-K/4}$, where
\begin{align}
F_{1,K}:=\left\{\|\Tho(t^{*})\|^2 + \int_{0}^{t^{*}} \psi(s)\left(\frac{1}{2}\|\nabla\Tho(s)\|^2
- 2\ra\rab\|\Tho(s)\|^2\right) ds - 4 -e^{-\frac{t^{*}}{4}}\|\Tho_0\|^2 \geq K\right\}\,.
\end{align}
Also define $F_{2,K}$ as $\tilde{E}_{2, K}$ in the proof of Proposition
\ref{lem:exp-bound} with $\Trs$ replaced by $\Tho$, and with the modified definition of $\psi(t)$ given in this proof.  As above we obtain
$\Prb (F_{2, K}) \leq e^{-K/4}$, and find that on $(F_{1,K}\cup F_{2,K})^{c}$ we have
\begin{align*}
\|\Tho(t^{*})\|^2 + \frac{e^{-\frac{t^{*}}{4}}}{2}\int_{0}^{t^{*}}\|\nabla\Tho(s)\|^2 \,ds
&\leq
 \ra\rab\int_{0}^{t^{*}}\psi(s)\|\Tho(s)\|^2 ds + 1 + e^{-\frac{t^{*}}{4}}\|\Tho_0\|^2 + K
 \\
&\leq
 C(\ra\rab + 1) (e^{-\frac{t^{*}}{4}}\|\Tho_0\|^2 + 1 + \rab^2 + K)
\,.
\end{align*}
That is, with
 $\tilde{C}_4 = C( \ra\rab +1)(1 + \rab^2)$, we conclude that
\begin{align*}
\Prb\left(\|\Tho(t^{*})\|^2 + \frac{e^{-\frac{t^{*}}{4}}}{2}\int_{0}^{t^{*}}\|\nabla\Tho(s)\|^2 \,ds -
 \tilde{C}_4 - C( \ra\rab +1)e^{-\frac{t^{*}}{4}}\|\Tho_0\|^2 \geq  C( \ra\rab +1) K \right) \leq 2e^{-K/4}.
\end{align*}
By using the formula \eqref{eq:cake} the inequality \eqref{eq:exp:theta} follows, completing the proof.
\end{proof}

\section{Existence of Invariant States}
\label{sec:ex}

In this section we apply the Krylov-Bogoliubov averaging procedure \cite{KryloffBogoliouboff1937}
to establish the existence of ergodic invariant measures \eqref{eq:B:eqn:vel:2}--\eqref{eq:theta} when $d=2$, and
for the infinite Prandtl system \eqref{eq:vel:inf}--\eqref{eq:theta:inf}
when $d=2$ or $d=3$.
We also prove the existence of a statistically invariant state for \eqref{eq:B:eqn:vel:2}--\eqref{eq:theta} when $d=3$
 by adapting the method of \cite{FlandoliGatarek1}.  That is, we provide the proof of Theorem \ref{thm:1b} at the end of this section.

\subsection{Invariant Measures}

\begin{Prop}
\label{prop:invar:meas:existence}
When $d =2$ there exists an ergodic invariant measure $\mu \in Pr(H)$ for the Markov semigroup $P_t$
corresponding to \eqref{eq:B:eqn:vel:2}--\eqref{eq:theta}.  In other words,
$\mu P_t = \mu$ for every $t \geq 0$ and if $P_t \chi_A = \chi_A$, $\mu$-almost surely then $\mu(A) \in \{0,1\}$.  Moreover
each invariant measure $\mu$ of \eqref{eq:B:eqn:vel:2}--\eqref{eq:theta} satisfies the exponential moment bound
\begin{align}
  \int_{H} \exp(\eta (\|\bfU\|^2+\|\theta\|_{L^p}^{2})) d \mu(U) \leq C\exp(\eta C_5) < \infty
  \label{eq:exp:your:mom:invar:measure}
\end{align}
for all $0<\eta < \eta_1$, $p\geq 2$, where $\eta_1 = \eta_1(p,\prN, \ra, \rab,\|\tilde{\sigma}\|)$ and $C_5 = C_5(p,\prN, \ra, \rab,\|\tilde{\sigma}\|)$.
These constants are independent
of $\prN$ if $\tilde{\sigma} = 0$.
\end{Prop}

Before we present the proof of Proposition \ref{prop:invar:meas:existence},
let us remark on an important consequence which appeared in our companion work \cite{FoldesGlattHoltzRichardsThomann2014b}.
\begin{Cor}
\label{cor:ref1}
Suppose $\tilde{\sigma}=0$ and $d=2$.  Let $\{\mu_{\prN}\}_{\prN\geq 2}$ denote any sequence of invariant measures for
\eqref{eq:B:eqn:vel:2}--\eqref{eq:theta} corresponding to increasing values of $\prN\geq 2$.
For any $p\geq 2$, there exists $\eta=\eta(p, \ra, \rab,\|\tilde{\sigma}\|)$, independent of $\prN$, such that
\begin{align}
\label{eq:unif:4}
\sup_{\prN\geq 2}\int_{H}\exp\left(\eta\left(\|\bfU\|^2 + \|\theta\|_{L^p}^{2}\right)\right)d\mu_{\prN}(\bfU,\theta)  < \infty.
\end{align}
\end{Cor}
\noindent

\begin{proof}[Proof of Proposition \ref{prop:invar:meas:existence}]
Suppose $d=2$, fix any $(\bfU_0,\Trs_0) \in H$, $T > 0$ and
let $\bfU(t) = \bfU(t, \bfU_0)$, $\Trs(t) = \Trs(t, \Trs_0)$ denote the solution of \eqref{eq:B:eqn:vel:2}--\eqref{eq:theta}.
Define the probability measures $\mu_T \in Pr(H)$ as
\begin{align}
  \mu_{T}(A) = \frac{1}{T} \int_{0}^{T} \Prb((\bfU,\Trs)(t) \in A) dt,
\qquad (A \in \mathcal{B}(H))\,.
\label{eq:KB-def}
\end{align}
From \eqref{eq:prb:est} we have
\begin{align}
\label{eq:invar:1}
  \sup_{T \geq 1} \frac{1}{T} \E \int_0^T \left( \| \nabla \bfU\|^2 + \|\nabla \Trs \|^2 \right)ds \leq C <\infty\,,
\end{align}
and hence with the Chebyshev inequality we infer that the collection $\{ \mu_T\}_{T \geq 1}$
is tight and thus weakly compact.  It follows immediately that any sub-sequence
converges to an invariant measure
for \eqref{eq:B:eqn:vel:2}--\eqref{eq:theta}; see e.g. \cite{ZabczykDaPrato1992}.

We next establish \eqref{eq:exp:your:mom:invar:measure} as follows.  Consider any invariant measure $\mu \in Pr(H)$
of \eqref{eq:B:eqn:vel:2}--\eqref{eq:theta}.  Fix any $R > 0$ and define
$\Phi_R : H \to \RR$ as
$\phi_R(U) = \exp(\eta (\| \bfU\|^2 + \|\Trs\|_{L^p}^2)) \wedge R$, where $U=(\bfU,\theta)$,
$\eta$ is given as in \eqref{eq:exp:your:mom:invar:measure} and $R > 0$.  Since $\phi_R$ is continuous and bounded on
$H$, and $\mu$ is
invariant, it follows that
\begin{align}
	\int_H \phi_R(U) d \mu(U)  = \int_H P_t \phi_R(U) d \mu (U)
	= \int_H \int_H \phi_R(\tilde{U}) P_t(U, d\tilde{U}) d \mu (U),
\label{eq:unif:1}
\end{align}
for any $t \geq 0$.
On the other hand, by applying \eqref{eq:exp:est:2} and \eqref{eq:gen:drift:diff:bnd:2}
we infer that, for any $U \in H$,
\begin{align}
\label{eq:unif:2}
\int_H \phi_R(\tilde{U}) P_t(U_0, d\tilde{U}) = \E \phi_R(U(t,U_0))
\leq C \exp\left(\eta \left(e^{ -\epsilon_1 t} \kappa_1 (\|u_0\|^2+\|\Trs_0\|_{L^p}^{2}) + C_5\right) \right),
\end{align}
where $\epsilon_1,\kappa_1, C_5>0$ depend on $\ra,\rab,p$ and $|\mathcal{D}|$ and $\prN$.
Note that these constants are independent of $\prN$ if $\tilde{\sigma}=0$.

Combining these observations we infer that, for any $t > 0$, $\rho > 0$, $R > 0$,
\begin{align}
\label{eq:unif:3}
   \int_H \phi_R(U) d \mu(U)
   	   &=
   \int_{B(\rho)} \int_{H} \phi_R(\tilde{U}) P_t(U, d\tilde{U}) d \mu (U) +
   \int_{B(\rho)^c} \int_{H} \phi_R(\tilde{U}) P_t(U, d\tilde{U}) d \mu (U)\notag \\
	&\leq
   C \exp\left(\eta \left(e^{ -\epsilon_1 t} \kappa_1 \rho^2 + C_5\right) \right) + R \mu( B(\rho)^c).
\end{align}
The desired result now follows by first taking $t \to \infty$ then $\rho \to \infty$ and finally $R \to \infty$
and using the monotone convergence theorem.

Working from \eqref{eq:exp:your:mom:invar:measure} we can establish further uniform moment bounds
in $H^1$.  Returning to \eqref{eq:u:ito}, \eqref{eq:theta:ito} we infer
\begin{align}
 2 \E (\|\nabla \bfU\|^{2} + \|\nabla\Trs\|^{2})
=& \E \left(2(\ra + \rab)\langle \Trs,u_2\rangle  +  \prN \|\tilde{\sigma}\|^{2} + 1\right) \notag\\
\leq& \E \left((\ra + \rab) (\| \Trs \|^2 + \|u_2 \|^2)  +  \prN \|\tilde{\sigma}\|^{2} + 1\right)
\label{eq:H1:uniform:invar:0}
\end{align}
for any stationary solution of \eqref{eq:B:eqn:vel:2}--\eqref{eq:theta}.  Taking
$\mathcal{I}$ to be the collection of invariant measures of $\{P_t\}_{t \geq 0}$,
and combining \eqref{eq:H1:uniform:invar:0} with \eqref{eq:exp:your:mom:invar:measure}
we infer 
\begin{align}
   \label{eq:H1:uniform:invar}
   \sup_{\mu \in \mathcal{I}} \int_{H} \| \nabla U \|^2 d \mu(U) < \infty.
\end{align}
So far we have only obtained the existence of an invariant measure, in particular    $\mathcal{I}$ is non-empty.  It is easy to see that $\mathcal{I}$ is convex and closed.
In view of \eqref{eq:H1:uniform:invar} we have furthermore that the collection $\mathcal{I}$ is tight
and hence compact.  We may thus infer the existence of an extremal
point in $\mathcal{I}$ from Krein-Milman theorem.    Since, cf. \cite{ZabczykDaPrato1992}, the
ergodic invariant measure consist of the extremal points of $\mathcal{I}$, we infer that $\mathcal{I}$
must contain an ergodic invariant measure, completing the proof of the first item.
\end{proof}

\begin{Prop}
\label{prop:6} The Markov semigroup corresponding to \eqref{eq:vel:inf}--\eqref{eq:theta:inf} possesses an
ergodic invariant measure $\mu\in Pr(H_2)$, and each invariant measure satisfies the exponential moment bound
\begin{align}
  \int_{H} \exp(\eta \|\Tho\|_{L^p}^2) d \mu(\Tho) \leq C\exp(\eta \tilde{C}_4) < \infty,
  \label{eq:exp:your:mom:invar:measure:2}
\end{align}
for all $0<\eta < \eta_2$, $p\geq 2$, where $\eta_2 = \eta_2(p,\ra, \rab) > 0$, $\tilde{C}_5 = \tilde{C}_5(p,\ra, \rab)$.
\end{Prop}

\begin{proof}[Proof of Proposition \ref{prop:6}]
By invoking Proposition \ref{prop:dde}, this requires a simple modification of the proof of Proposition \ref{prop:invar:meas:existence},
and we omit the details.
\end{proof}

\subsection{Existence of Statistically Invariant States}

In three space dimensions we establish the existence of statistically invariant states.  Namely, we rephrase and prove Theorem \ref{thm:1b}
as follows.
\begin{Prop}
\label{prop:5}
When $d = 3$ there exists a measure $\mu \in Pr(H)$ and a corresponding (Martingale) solution of \eqref{eq:B:eqn:vel:2}--\eqref{eq:theta}
which is stationary in time, namely $\mu_t(\cdot) = \Prb(U(t) \in \cdot)$ is identically equal to $\mu$.  Moreover $\mu$ satisfies an
exponential moment bound as in \eqref{eq:exp:your:mom:invar:measure}.
\end{Prop}

As above, we mention a useful corollary of Proposition \ref{prop:5} which appeared in the manuscript
\cite{FoldesGlattHoltzRichardsThomann2014b}.
\begin{Cor}
\label{cor:2}
Suppose $\tilde{\sigma}=0$ and $d=3$.  For any $p\geq 2$, the system \eqref{eq:B:eqn:vel:2}--\eqref{eq:theta} possesses a sequence of statistically invariant states
$\{\mu_{\prN}\}_{\prN\geq 2}$ such that \eqref{eq:unif:4} is satisfied.
\end{Cor}
\noindent One can prove Corollary \ref{cor:2} by combining \eqref{eq:exp:your:mom:invar:measure} with
$\prN$-uniform bounds on the $\|\theta\|_{L^p}^{2}$ term in \eqref{eq:unif:4}.  These latter bounds follow by using \eqref{eq:gen:drift:diff:bnd:2} and
mimicking the proof of \eqref{eq:exp:your:mom:invar:measure} in the case $d=3$ (see below).

\begin{proof}[Proof of Proposition \ref{prop:5}]
We establish the existence of a statistically invariant state using a modified Galerkin
truncation scheme.  Let $P_N$ be the projection onto the first $N$ eigenfunctions of the
Stokes operator, which are of course divergence free. Define
\begin{align}
  &d \bfU^N + P_N( \bfU^N \cdot \nabla \bfU^N )= \prN P_N \Delta \bfU^N dt + \prN \ra P_N \hate \Trs^N dt
  + \prN \sum_{k = 1}^{N_1} P_N \tilde{\sigma}_k d\tilde{W}^k, \quad  \bfU^N(0) = P_N \bfU_0 \label{eq:B:eqn:trunc:1}\\
  &d \Trs^N + \bfU^N \cdot \nabla \Trs^N dt = u_d^N dt + \Delta \Trs^N dt + \sum_{k =1}^{N_2} \sigma_k dW^k,  \quad \Trs^N(0) = \Trs_0.
  \label{eq:B:eqn:trunc:2}
\end{align}
Observe that for each $N$ there is a unique pathwise solution $U^N(\cdot, U_0)$ which depends continuously on its initial condition $U_0 = (\bfU_0, \Trs_0) \in H$.
Take $\{P^N\}_{t \geq 0}$ to be the associated Markov semigroup.  As in Proposition~\ref{lem:exp-bound} we have
\begin{align}
	\Prb \! \left(\sup_{t\geq 0}\left(\|\bfU^N\! (t)\|^{2}
			\! + \! \frac{\prN}{2}\|\Trs^N \! (t)\|^2
                         +\frac{\prN}{2} \! \int_{0}^{t} \! \! \left(\|\nabla \bfU^N \|^2 + \frac{1}{2}\|\nabla\Trs^N \|^2 \right)\! ds - C_2t - C_1 \! \right)\geq C_3 K\!\right)
                         \! \leq 3 e^{-\gamma K},
                         \label{eq:mod:gal:exp:bnd:1}
\end{align}
where the constants $C_1, C_2, C_3$ have the same meaning as in \eqref{eq:def-c1}--\eqref{eq:def-c3}.  To see this observe that
we have only used the cancellation property in the nonlinear term to analyze the momentum equation.   We have deliberately avoided
truncating the equation for $\Trs^N$ so that $\Trs^N$ still
satisfies drift-diffusion equation \eqref{eq:bous:drift:diff:stoch:sys:1}.

Making use of \eqref{eq:mod:gal:exp:bnd:1} we may now obtain the existence of an invariant measure $\mu^N$ for each $N$ by implementing the Krylov-Bogoliubov procedure
and arguing precisely as in the $d = 2$ case.
We observe moreover that
\begin{align}
  \int_H \exp(\eta (\|\bfU\|^2 + \|\Trs\|_{L^p}^{2})) d \mu^N(U) \leq \exp(\eta C_1)\,,
  \label{eq:mod:gal:exp:bnd:2}
\end{align}
where, crucially, the constant $C_1$ in this upper bound is independent of $N$.  Here again the arguments in the proof of \eqref{eq:exp:your:mom:invar:measure} pass through virtually line by line to the present case.

To a sequence of invariant measures $\{ \mu_N\}_{N \geq 1}$ we may now associate a a
sequence of stationary solutions $\bar{U}^N$ of \eqref{eq:B:eqn:trunc:1}--\eqref{eq:B:eqn:trunc:2}.    An involved
limiting procedure very similar to e.g. \cite{FlandoliGatarek1} (see also \cite{KuksinShirikian12,GlattHoltzSverakVicol2013})
can be used to pass to a limit in this class of stationary solutions.  We briefly sketch some details.

By repeating the computations
leading to \eqref{eq:mod:gal:exp:bnd:1}
with the stationary solutions $\bar{U}^N$ we obtain that
\begin{align}
   \bar{U}^N \textrm{ is uniformly bounded in } L^p( \Omega; L^2([0,t^*]; V) \cap L^\infty([0,t^*];H))
   \label{eq:uniform:bnd:stat:sols:1}
\end{align}
for any $p \geq 1$ and any $t^* > 0$.  On the other hand writing \eqref{eq:B:eqn:trunc:1}--\eqref{eq:B:eqn:trunc:2}
as an abstract evolution equation on $H$ we may write $\bar{U}^N(t) = \bar{U}^N(0) + \int_0^t F_N(\bar{U}^N) ds + \sigma_N W(t)$
where $F_N$ is a suitable abstract operator; see e.g. \cite{ConstantinFoias88, Temam2001} for precise details.
Define now $X = V \cap (H^3(\DD))^4$ let $X^*$ be its dual, relative to $H$.   For any $\alpha \in (0, 1/2)$ and $p > 1$ we have
that
\begin{align*}
   \E \|\bar{U}^N\|_{W^{\alpha,p}([0,t^*]; X^*)}^p
   \leq C \E \left\|\bar{U}^N(0) + \int_0^{\cdot} F_N(\bar{U}^N) ds \right\|_{W^{1,p}([0,t*]; X^*)}^p +C\E \|\sigma_N W(t) \|_{W^{\alpha,p}([0,t^*]; H)}^p
\end{align*}
here $W^{\alpha,p}([0,T]; X^*)$ is the Sobolev-Slobodeckij space.
Since $\| F_N(U)\|_{X^*} \leq C(\|U\|^2 + 1)$
for any $U \in H$, where the constant $C$ does not depend on $N$ we infer that
\begin{align*}
   \E \left\|\bar{U}^N(0) + \int_0^{\cdot} F_N(\bar{U}^N) ds \right\|_{W^{1,p}([0,t^*]; X^*)}^p \leq C \E \left(  \sup_{s \in [0,t^*]} \| \bar{U}^N \|^{2p} + 1  \right).
\end{align*}
Moreover, given the standard time regularity properties of Brownian motions we obtain the uniform bound
\begin{align*}
  \sup_{N \geq 1} \E \|\sigma_N W(t) \|_{W^{\alpha,p}([0,t^*]; H)}^p  < \infty \,.
\end{align*}
Combining these bounds with \eqref{eq:uniform:bnd:stat:sols:1} we now obtain that
\begin{align}
   \bar{U}^N \textrm{ is uniformly bounded in } L^p( \Omega; W^{\alpha, p}([0,T]; X^*))
   \label{eq:uniform:bnd:stat:sols:2}
\end{align}
for any $\alpha < 1/2$ and $p > 1$.

Take $Y = V \cap (H^4(\DD))^4$ and $Y^*$ its dual relative to $H$.    We now show that
the collection of measures $\{\mu_N\}_{N \geq 1}$ given the laws of $\bar{U}^N$ and defined
on the path-space $C([0,t^*]; Y^*)$ are compact.
To this end recall that we have the compact embeddings
\begin{align}
   L^2([0,t^*]; V) \cap W^{1/4, 2}([0,t^*]; X^*) \subset \subset L^2([0,t^*];H), \quad
   W^{1/4, 8}([0,t^*]; X^* ) \subset \subset C([0,t^*]; Y^*)
   \label{eq:compact:embeddings}
\end{align}
which are variations on the Aubin-Lions and Arzel\`a-Ascoli compactness theorems respectively; see
\cite{FlandoliGatarek1} and also \cite{DebusscheGlattHoltzTemam1} for further details.
Combining the \eqref{eq:compact:embeddings} with \eqref{eq:uniform:bnd:stat:sols:1}, \eqref{eq:uniform:bnd:stat:sols:2}
we obtain that, after passing to a subsequence $\mu_N$ converges weakly in $L^2([0,t^*];H) \cap C([0,t^*]; Y^*)$.

By changing the stochastic basis we can, by Skorokhod's embedding theorem,
assume that $\bar{U}^N$ converges almost surely in $L^2([0,t^*];H)$, and in $C([0,t^*]; Y^*)$, to an element $\bar{U}$,
while maintaining the uniform bound \eqref{eq:uniform:bnd:stat:sols:1}.
Additionally, since $\bar{U}^N$ is stationary, in view of the convergence in $C([0,t^*]; Y^*)$, it follows that $\bar{U}$ is stationary too.
In order to show that $\bar{U}$ is a Martingale solution of the system \eqref{eq:B:eqn:vel:2}--\eqref{eq:theta}, one needs to obtain an
appropriate stochastic basis in conjunction with this
limiting procedure.  Here several approaches have been developed and can be used, see \cite{ZabczykDaPrato1992,FlandoliGatarek1,Bensoussan1,
DebusscheGlattHoltzTemam1,hofmanova2012}.

Finally we remark that the measure $\mu$ we obtain as the weak limit of $\{\mu^N\}$ satisfies \eqref{eq:exp:your:mom:invar:measure}.  Indeed
by \eqref{eq:mod:gal:exp:bnd:2}, $\int_H \exp(\eta (\|\bfU\|^2 + \|\Trs\|_{L^p}^2)\wedge R d \mu^N(U) \leq \exp(nC_1)$ for every $R > 0$, where the constant $C_1$ is independent of $R$.
using first weak convergence and then the monotone convergence theorem we can now pass to a limit first in $N$ and then in $R$. With this the proof of
Proposition~\ref{prop:5}, and thus of Theorem~\ref{thm:1b} is complete.
\end{proof}

\section{Unique Ergodicity Results}
\label{sec:4}

We proceed to prove the main results on unique ergodicity in this work, Theorem \ref{thm:1} and Theorem \ref{thm:2}.
We will begin by presenting an abstract framework for proving unique ergodicity by asymptotic coupling following \cite{HairerMattinglyScheutzow2011}.
This framework will serve to simplify our analysis.  Recall that the strategy for proving uniqueness
via this framework was
expanded upon for multiple examples in forthcoming work of a subset of the coauthors in collaboration with Mattingly,
\cite{GlattHoltzMattinglyRichards2015}.

\subsection{Abstract Results from Ergodic Theory}

To prove our ergodic theorems we will make use of the following abstract result from \cite{HairerMattinglyScheutzow2011}.
Let $P$ be a Markov transition function on a Polish space $(H, \rho)$.    Denote the collection of
all Borel probability measures on $H$ by $Pr(H)$.  Take $P_{\infty}$ to
be the associated transition function on the path space of infinite one-sided sequences $H_{\infty}=H^\NN$,
and for $\mu\in Pr(H)$, let $\mu P_{\infty}\in Pr (H_{\infty})$ be the measure defined by $\int_{H_\infty}P_{\infty}(x,\cdot)d\mu(x)$.
Given $\mu_{1}, \mu_{2} \in \Pr(H)$, consider
\begin{align*}
  \tilde{\mathcal{C}}_{\infty}(\mu_{1}, \mu_{2}) := \{ \Gamma \in Pr (H_{\infty} \times H_{\infty}): \Pi_{i}^\# \Gamma \ll \mu_{i} P_{\infty} \textrm{ for each } i \in \{1, 2\} \}\,,
\end{align*}
where $\Pi_{i}$ is the projection onto the $i^{\textrm{th}}$ coordinate and $f^\# \mu$ is a push forward of the measure $\mu$, that is, $f^\# \mu(A) = \mu(f^{-1}(A))$.
We will also denote $D := \{(x,y) \in H_{\infty}\times H_{\infty}: \lim_{n \to \infty} \rho(x_{n},y_{n}) = 0 \}$.
\begin{Thm}[\cite{HairerMattinglyScheutzow2011}]\label{thm:HMS2011}
 Suppose there exists a Borel measurable set $\mathcal{A} \subset H$ such that
\begin{itemize}
\item[(i)] for any $P$ invariant Borel probability measure $\mu$, $\mu(\mathcal{A}) > 0$,
\item[(ii)] there exists a measurable map $\Gamma: \mathcal{A} \times \mathcal{A} \to Pr(H_{\infty} \times H_{\infty})$
such that,  for all $x, y \in \mathcal{A}$,  $\Gamma(x,y) \in  \tilde{\mathcal{C}}_{\infty}(\delta_{x},\delta_{y})$
and $\Gamma(x,y)(D) > 0$,
\end{itemize}
then there is at most one invariant probability measure for $P$.
\end{Thm}

\subsection{Proof of Theorem~\ref{thm:1}}

\begin{proof}[Proof of Theorem \ref{thm:1}]
Let $P_\infty $  be the transition function on $H_{\infty}$, with the product topology,  corresponding to equations \eqref{eq:B:eqn:vel:2}--\eqref{eq:theta}
evaluated at integer times.
For a given $\Utot=(\bfU,\Trs)$ satisfying \eqref{eq:B:eqn:vel:2}--\eqref{eq:theta} and $\tilde{\Utot}_0=(\tilde{\bfU}_0,\tilde{\Trs}_0)$,
define $\tilde{\Utot}=(\tilde{\bfU},\tilde{\Trs})$ as the solution of
\begin{align}
  \frac{1}{\prN}(d \tilde{\bfU} + \tilde{\bfU} \cdot \nabla \tilde{\bfU} dt) + \nabla \tilde{p} dt &= \Delta \tilde{\bfU} dt + \ra \hate\tilde{\Trs} dt
  + \sum_{k=1}^{N_1} \tilde{\sigma}_k d\tilde{W}^k
  \label{eq:B:eqn:vel:sft}\\
  & \quad \quad - \lambda_1 P_{N_1} (\tilde{\bfU}- \bfU)\mathbbm{1}_{t<\tau_{R}} dt, \quad \nabla \cdot \tilde{\bfU} = 0,\notag \\
  d \tilde{\Trs} + \tilde{\bfU} \cdot \nabla \tilde{\Trs} dt &= \rab \tilde{u}_d dt+  \Delta \tilde{\Trs} dt +  \sum_{k=1}^{N_2}
  \sigma_k dW^k -
  \lambda_2 P_{N_2} (\tilde{\Trs}- \Trs)
  \mathbbm{1}_{t<\tau_{R}}dt,
  \label{eq:B:eqn:temp:sft}
\end{align}
with initial data $\tilde\Utot_0.$
Here $\mathbbm{1}_{t<\tau_{R}}$ is the characteristic function of $\{t<\tau_R\}$, where
\begin{align}
\label{eq:stop:1}
\tau_R = \inf\left\{t>0:\int_{0}^{t}\left(\|\sigma^{-1}P_{N_1}(\tilde{\bfU}- \bfU)\|^2 + \|\tilde{\sigma}^{-1}P_{N_2}(\tilde{\Trs}- \Trs)\|^2\right)ds \geq R
\right\}.
\end{align}
Let $\tilde{P}_\infty$ be the transition function on $H_{\infty}$ associated to \eqref{eq:B:eqn:vel:sft}--\eqref{eq:B:eqn:temp:sft},
and as above, for fixed $\mu\in Pr(H)$, let $\mu P_{\infty}\in Pr (H_{\infty})$ denote the probability measure
$\mu \tilde{P}_{\infty}(\cdot)= \int_{H_\infty}\tilde{P}_{\infty}(x,\cdot)d\mu(x)$.
This proof will proceed by applying Theorem \ref{thm:HMS2011} with $\mathcal{A}=H$ and $\Gamma= \Gamma_{U_0,\tilde{U}_0} :=
\delta_{U_0}P_\infty \times  \delta_{\tilde{U}_0}\tilde{P}_\infty $ for each $U_0,\tilde{U}_0 \in H$.  In order to do this, we first show that, with fixed initial conditions,
the laws of solutions to \eqref{eq:B:eqn:vel:2}--\eqref{eq:theta} and
\eqref{eq:B:eqn:vel:sft}--\eqref{eq:B:eqn:temp:sft} are equivalent.

Recall that $\mathbf{W} = (\tilde{W},W) : \Omega \times (0, \infty) \to \RR^{N_1 + N_2}$ is a standard Wiener process on
a filtered probability space $(\Omega, \mathcal{F}, (\mathcal{F}_t)_{t \geq 0}, \Prb)$.
We define the ``Girsanov shift''
\begin{align*}
a(t) := \begin{pmatrix} \tilde{\sigma}^{-1}\lambda_1 P_{N_1}(\tilde{\bfU}- \bfU) \\
\sigma^{-1}\lambda_2 P_{N_2}(\tilde{\Trs}- \Trs)\end{pmatrix}\mathbbm{1}_{t<\tau_{R}},
\end{align*}
and let
\begin{equation}
\mathbf{W}_0(t) = \mathbf{W}(t) - \int_{0}^{t}a(s) ds.
\end{equation}
We also consider \begin{align}
D(t) := \exp\left(\int_{0}^{t}a(s)\cdot d\mathbf{W}- \frac{1}{2}\int_{0}^{t}\|a(s)\|^{2}ds\right).
\end{align}
We can easily check that $D(t)$ is a continuous martingale, and since $a(t)$ satisfies the Novikov condition (see Proposition 1.15 of \cite{RevuzYor1999}),
\begin{align*}
 \E \exp\left(\frac{1}{2}\int_{0}^{\infty}\|a(s)\|^{2}ds\right)
= \E \exp\left(\frac{1}{2}\int_{0}^{\tau_R}\|a(s)\|^{2}ds\right) \leq e^{\frac{1}{2}R}<\infty,
\end{align*}
it follows that $\{D(t)\}_{t\geq 0}$ is uniformly integrable.  By the Girsanov theorem, (see Propositions 1.1-1.4 of \cite{RevuzYor1999})
there is a probability measure $\mathbb{Q}$ on $(\Omega, \mathcal{F}, (\mathcal{F}_t)_{t \geq 0})$ such that
the Radon-Nikodym derivative
of $\mathbb{Q}$ with respect to $\Prb$ on $\mathcal{F}_{t}$ is $D(t)$, and under $\mathbb{Q}$, $\mathbf{W}_0$ is a standard Wiener process.
Moreover, it holds that $\Prb \sim \mathbb{Q}$ on $\mathcal{F}_{\infty}=\sigma(\cup_{t\geq 0}\mathcal{F}_{t})$.

Let $U=U(t,U_0,\mathbf{W}(t))$ and $\tilde{U}=\tilde{U}(t,U_0,\mathbf{W}(t))$
denote the solutions to \eqref{eq:B:eqn:vel:2}--\eqref{eq:theta} and \eqref{eq:B:eqn:vel:sft}--\eqref{eq:B:eqn:temp:sft}, respectively, at time $t\geq 0$,
with initial condition $U_0$, and with stochastic forcing $\mathbf{W}$.  Then by the uniqueness of solutions (see e.g. Theorem 2.4.6 of \cite{KuksinShirikian12})
we have that, almost surely,
\begin{align}
\label{eq:g:3}
\tilde{U}(\cdot,U_0,\mathbf{W}(\cdot))=U(\cdot,U_0,\mathbf{W}_0(\cdot)).
\end{align}
Next recall that, for each fixed initial condition $U_0$, the solutions of \eqref{eq:B:eqn:vel:2}--\eqref{eq:theta}
and \eqref{eq:B:eqn:vel:sft}--\eqref{eq:B:eqn:temp:sft} induce measurable maps, which we denote by $\Phi_{U_0}$ and $\tilde{\Phi}_{U_0}$, respectively,
from $(\Omega,\mathcal{F},\mathbb{P})$ into $C([0,\infty);H)$.  It follows that the law of solutions to \eqref{eq:B:eqn:vel:2}--\eqref{eq:theta} is given by $(\Phi_{U_0})_{\#}\Prb$,
and similarly, the law associated to \eqref{eq:B:eqn:vel:sft}--\eqref{eq:B:eqn:temp:sft} is $(\tilde{\Phi}_{U_0})_{\#}\Prb$.  By \eqref{eq:g:3},
for any $\mathcal{F}_{\infty}$-measurable $G\subset C([0,\infty);H)$,
\begin{align}
(\tilde{\Phi}_{U_0})_{\#}\Prb(G) = \Prb\left(\tilde{U}(\cdot,U_0,\mathbf{W}(\cdot))\in G\right) = \Prb\left(
\vphantom{\tilde{U}}U(\cdot,U_0,\mathbf{W}_0(\cdot))\in G\right),
\label{eq:g:4}
\end{align}
whereas, since $\mathbf{W}_0$ is a standard Wiener process under $\mathbb{Q}$,
\begin{align}
(\Phi_{U_0})_{\#}\Prb(G) = \Prb\left(\vphantom{\tilde{U}}U(\cdot,U_0,\mathbf{W}(\cdot))\in G\right)
= \mathbb{Q}\left(\vphantom{\tilde{U}}U(\cdot,U_0,\mathbf{W}_0(\cdot))\in G\right).
\label{eq:g:5}
\end{align}
By combining \eqref{eq:g:4} and \eqref{eq:g:5} with $\Prb \sim \mathbb{Q}$ on $\mathcal{F}_{\infty}$, we conclude that
$(\tilde{\Phi}_{U_0})_{\#}\Prb \sim (\Phi_{U_0})_{\#}\Prb$.  That is, with the same initial conditions, the laws of solutions to
\eqref{eq:B:eqn:vel:2}--\eqref{eq:theta}
and \eqref{eq:B:eqn:vel:sft}--\eqref{eq:B:eqn:temp:sft} are equivalent on $C([0,\infty);H)$.

Taking $\mathcal{A}=H$ and setting $\Gamma= \Gamma_{U_0,\tilde{U}_0} :=
\delta_{U_0}P_\infty \times  \delta_{\tilde{U_0}}\tilde{P}_\infty $ for each $\Utot_0,
\tilde{\Utot}_0\in H$, we proceed to check that the assumptions of Theorem \ref{thm:HMS2011} are satisfied.
Clearly $H$ is a Polish space of nonzero measure for every invariant measure.
From what was proven above, we have $  \Pi_2^{\#}\Gamma_{\Utot_0,\tilde{\Utot}_0} = \delta_{\tilde{\Utot}_0} \tilde{P}_{\infty}
\ll \delta_{\tilde{\Utot}_0}P_{\infty}$, and
it follows that $\Gamma_{\Utot_0,\tilde{\Utot}_0} \in
 \tilde{\mathcal{C}}_{\infty}(\delta_{{\Utot}_0},\delta_{\tilde{\Utot}_0})$.
It remains to show that $\Gamma_{\Utot_0,\tilde{\Utot}_0}(D) > 0$.

Set $\Vtot =(\bfV,\Tdiff) := \tilde{\Utot} - \Utot$.  We obtain that
\begin{align}
  \frac{1}{\prN}&(\pd{t} \bfV + \tilde{\bfU} \cdot \nabla \bfV + \bfV \cdot \nabla \bfU) + \nabla p  = \Delta \bfV  + \ra \hate \Tdiff
  - \lambda_1  P_{N_1} \bfV\mathbbm{1}_{t<\tau_{R}},
  \label{eq:B:eqn:vel:diff}\\
  &\pd{t} \Tdiff + \tilde{\bfU} \cdot \nabla \Tdiff + \bfV \cdot \nabla \Trs
  = \rab v_d +\Delta \Tdiff - \lambda_2 P_{N_2} \Tdiff\mathbbm{1}_{t<\tau_{R}}.
  \label{eq:B:eqn:temp:diff}
\end{align}
Standard energy estimates yield
\begin{align*}
 \frac{1}{2} \frac{d}{dt} \|\bfV\|^2 +  \prN \|\nabla\bfV\|^2 +
\lambda_1 \prN \mathbbm{1}_{t<\tau_{R}}\| P_{N_1} \bfV \|^2=  \int_{\DD} (\prN \ra  \Tdiff v_d- \bfV \cdot \nabla \bfU \cdot \bfV) dx
\end{align*}
and
\begin{align*}
   \frac{1}{2} \frac{d}{dt} \|\Tdiff\|^2+   \|\nabla \Tdiff\|^2 +  \lambda_2 \mathbbm{1}_{t<\tau_{R}}
    \|P_{N_2} \Tdiff\|^2= \int_{\DD} (\rab v_d \Tdiff - \Tdiff\bfV \cdot \nabla \Trs  )dx.
\end{align*}
That is,
\begin{align}
\frac{1}{2}  \frac{d}{dt} (\|\bfV\|^2 + \|\Tdiff\|^2) +&  \prN \| \nabla \bfV\|^2 +
\| \nabla\Tdiff\|^2 + \lambda_1 \prN \mathbbm{1}_{t<\tau_{R}} \|  P_{N_1}\bfV \|^2 +
 \lambda_2 \mathbbm{1}_{t<\tau_{R}}\| P_{N_2}\Tdiff\|^2 \notag \\
  &\leq
    \int_{\DD} \left( (\rab +  \prN \ra) |\Tdiff v_d| + | \bfV \cdot \nabla \bfU \cdot \bfV| +
    |\Tdiff \bfV \cdot \nabla \Trs |  \right)dx.
\label{eq:energy-est}
\end{align}
We estimate the last two terms on the right-hand side of \eqref{eq:energy-est} as follows
\begin{align*}
\int_{\DD} \left(| \bfV \cdot \nabla \bfU \cdot \bfV| +
    |\Tdiff \bfV \cdot \nabla \Trs |  \right)dx
  &\leq  \| \bfV\|_{L^4}^2 \|\nabla \bfU\|
    +  \|\Tdiff\|_{L^4}\| \bfV\|_{L^4} \|\nabla \Trs\| \\
    &\leq C\left(\| \bfV\| \|\nabla \bfV\|
    \|\nabla\bfU\| +  \| \Tdiff\|^{1/2} \|\nabla \Tdiff \|^{1/2}\| \bfV\|^{1/2}
    \|\nabla \bfV\|^{1/2}\|\nabla \Trs\| \right)
    \\
     &\leq \frac{\prN}{2}\|\nabla \bfV\|^{2} + \frac{1}{2}\|\nabla \Tdiff\|^{2} + \frac{C}{\prN}\| \bfV\|^2\| \nabla\bfU\|^2 +
     \frac{C}{\sqrt{\prN}} \left(\| \Tdiff\|^2+ \|\bfV\|^2\right)
  \|\nabla \Trs\|^2.
\end{align*}
From this estimate and \eqref{eq:energy-est} we obtain
\begin{align}
\frac{1}{2}  \frac{d}{dt} (\|\bfV\|^2 + \|\Tdiff\|^2) +&  \frac{1}{2} (\prN \| \nabla \bfV\|^2 + \| \nabla\Tdiff\|^2)
+ \prN \lambda_1 \mathbbm{1}_{t<\tau_{R}} \| P_{N_1} \bfV \|^2 +
 \lambda_2 \mathbbm{1}_{t<\tau_{R}}\| P_{N_2} \Tdiff\|^2 \notag \\
  &\leq
   (\rab +  \prN \ra) \int_{\DD}  |\Tdiff v_d|dx + \frac{C}{\sqrt{\prN}}(\| \bfV\|^2\| \nabla\bfU\|^2
     + \left(\| \Tdiff\|^2+ \|\bfV\|^2\right)
  \|\nabla \Trs\|^2 ).
\label{eq:energy-est-2}
\end{align}

We proceed to establish part (i) of Theorem \ref{thm:1}.  For this we consider, for each $K>0$,
the events
$$
\tilde{E}_K :=
\left\{\sup_{t\geq 0}\left(\|\bfU(t)\|^{2} + \frac{\prN}{2}\|\Trs(t)\|^2 +\frac{\prN}{2}\int_{0}^{t}
\left(\|\nabla\bfU(s)\|^2 + \frac{1}{2}\|\nabla\Trs(s)\|^2\right) ds -  C_1-C_2t
\right) \leq C_3 K\right\}\,,
$$
where $C_1, C_2, C_3$ are the constants defined in \eqref{eq:def-c1}--\eqref{eq:def-c3}.
By Proposition \ref{lem:exp-bound} we can choose $K = K(\prN, \ra) > 0$ such that $\Prb (\tilde{E}_K) > 0$.
We first show that for already fixed $K > 0$ there exists
$R = R(K, \|\bfU_0\|, \|\Trs_0\|,\|\tilde{\sigma}\|,\prN,\ra,\rab)$, such that $\tau_{R}=\infty$ on $\tilde{E}_K$.
Indeed, assume $\tau_{R}$ is finite for some element of $\tilde{E}_K.$  Then for $t<\tau_{R}$, we obtain
\begin{equation}
\lambda_1\mathbbm{1}_{t<\tau_{R}} \| P_{N_1}  \bfV \|^2 = \lambda_1
\| P_{N_1}  \bfV \|^2\,
\end{equation}
Moreover, for any $\lambda_1 > 0$ (specified below), we take integer $N_1 \geq \sqrt{2\lambda_{1}/C}$, where $C$ is a constant depending only on the domain, such that
\begin{align}\label{eq:ipo}
\lambda_1 \|\bfV\|^2  \leq \lambda_1 \|P_{N_1}\bfV\|^2 + \frac{C\lambda_1}{N_1^{2}}\|\nabla Q_{N_1}\bfV\|^2 \leq \lambda_1 \|P_{N_1}\bfV\|^2 + \frac{1}{2}\|\nabla\bfV\|^2\,,
\end{align}
where we used $\|\nabla Q_{N} v\|^2 \geq \kappa_N \|Q_N v\|^2$, with $\kappa_N$ being the $N^{\textrm{th}}$ eigenvalue of the Stokes
operator with the Dirichlet boundary conditions, and $\kappa_N \approx N^2$.

Similarly, for any $\lambda_2 > 0$, we take $N_2 \geq \sqrt{2\lambda_{2}/C}$
and denote $\lambda:= (\prN\lambda_1)\wedge \lambda_2$.
Then, we have
\begin{align}
\frac{1}{2}  \frac{d}{dt} (\|\bfV\|^2 + \|\Tdiff\|^2) +&
 \lambda ( \| \bfV \|^2 +
  \| \Tdiff\|^2 ) \notag \\
  &\leq
    (\rab +  \prN \ra)\int_{\DD}  |\Tdiff v_d|dx + \frac{C}{\sqrt{\prN}} (\| \bfV\|^2\| \nabla\bfU\|^2
     + \left(\| \Tdiff\|^2+ \|\bfV\|^2\right)
  \|\nabla \Trs\|^2).
\label{eq:energy-est-3}
\end{align}
Using
\begin{align*}
 \int_{\DD} |\Tdiff v_d|dx \leq
\| \Tdiff\|^2 + \|\bfV\|^{2},
\end{align*}
we obtain, under the condition $\lambda \geq 2(\rab + \prN\ra)$,
\begin{align}
\frac{d}{dt} (\|\bfV\|^2 + \|\Tdiff\|^2) +  \lambda\left( \| \bfV \|^2 +
 \| \Tdiff\|^2\right)
  & \leq   \frac{C}{\sqrt{\prN}} \left(\| \Tdiff\|^2+ \|\bfV\|^2\right)
  \left(\| \nabla\bfU\|^2 + \|\nabla \Trs\|^2\right)\,.
\label{eq:energy-est-5}
\end{align}
Then by Gr\"{o}nwall's inequality,
\begin{align*}
\|\bfV(t)\|^2 + \|\Tdiff(t)\|^2 \leq (\|\bfV_{0}\|^2 + \|\Tdiff_{0}\|^2)\exp\left(\frac{C}{\sqrt{\prN}} \int_{0}^{t}(\|\nabla \bfU(s)\|^2 + \frac{1}{2}\|\nabla \Trs(s)\|^2)ds-\lambda t\right)
\end{align*}
and thus on $\tilde{E}_{K}$,
\begin{align}
\label{eq:erg:1}
\|\bfV(t)\|^2 + \|\Tdiff(t)\|^2 \leq (\|\bfV_{0}\|^2 + \|\Tdiff_{0}\|^2)
\exp\left(\frac{C}{\prN^{3/2}}(C_3 K + C_1 + C_2t) - \lambda t\right).
\end{align}
It follows that, for $t<\tau_R$, by fixing
$\lambda>0$ such that $\lambda > \frac{ C C_2}{\prN^{3/2}}$, we have
\begin{align*}
\int_{0}^{t}&\left(\|\lambda_1\tilde{\sigma}^{-1}P_{N_1}\bfV(s)\|^2
+ \|\lambda_2\sigma^{-1}P_{N_2}\Tdiff(s)\|^2\right)ds
\\
&\leq \tilde{C\lambda}(\|\bfV_{0}\|^2 + \|\Tdiff_{0}\|^2)
\exp\left(\frac{C}{\prN^{3/2}}(C_3 K + C_1)\right)\int_{0}^{t}\exp\left(\left(\frac{C C_2}{\prN^{3/2}}-\lambda\right)s\right)ds
\\
&\leq  \tilde{C}(\|\bfV_{0}\|^2 + \|\Tdiff_{0}\|^2)
\exp\left(\frac{C}{\prN^{3/2}}(C_3 K + C_1)\right),
\end{align*}
where $\tilde{C} = \tilde{C}(\prN,\|\tilde{\sigma}\|)$ is, in particular, independent of $R$.
Choose $R>0$ such that
\begin{align}
R > 2\tilde{C}(\|\bfV_{0}\|^2 + \|\Tdiff_{0}\|^2)
\exp\left(\frac{C}{\prN^{3/2}}(C_3 K + C_1)\right).
\label{eq:erg:2}
\end{align}
Then, for $0<t<\tau_{R}$, we have
\begin{align*}
\int_{0}^{t}&\left(\|\lambda_1\tilde{\sigma}^{-1}P_{N_1}\bfV(s)\|^2 + \|\lambda_2\sigma^{-1}P_{N_2}\Tdiff(s)\|^2\right)ds < \frac{R}{2},
\end{align*}
which contradicts the assumption that $\tau_{R}$ is finite.

It follows that for these choices of $K$, $R$, and $\lambda,$ one has that on $\tilde{E}_K$,
the estimate
\begin{align*}
\|\bfV(t)\|^2 + \|\Tdiff(t)\|^2 \leq C\exp\left(\left(\frac{C C_2}{\prN^{3/2}}-\lambda\right)t\right)
\end{align*}
 holds for all $t>0$, and therefore
 \begin{align}
 \tilde{E}_{K}\subset\left\{\{(\Utot(n,\Utot_0),
 \tilde{\Utot}(n,\tilde{\Utot}_0))\}_{n\in\NN}\in D \right\},
 \label{eq:diag_estimate}
 \end{align}
 where $D=\{(v,w)\in H_{\infty}\times H_{\infty}:
 \lim_{n\rightarrow \infty}\|v_{n}-w_n\|=0\}$.  Moreover, from \eqref{eq:diag_estimate}, Proposition
 \ref{lem:exp-bound}, and our choice of $K$, it follows that
 \begin{align}
 \label{eq:diag_2}
 \Gamma_{\Utot_0,\tilde{\Utot}_0}(D)=P\left(\{(\Utot(n,\Utot_0),
 \tilde{\Utot}(n,\tilde{\Utot}_0))\}_{n\in\NN}\in D \right)
 \geq P(\tilde{E}_{K})>0 \,.
 \end{align}
By Theorem \ref{thm:HMS2011} the proof of part (i) of Theorem \ref{thm:1} is complete.

We proceed to establish part (ii) of Theorem \ref{thm:1}, for which we will take $\lambda_1 =0$.

Once again, with an appropriate choice of $R$, we will show that $\tau_{R}=\infty$ on $\tilde{E}_K.$
Indeed, assume $\tau_{R}$ is finite for some element of $\tilde{E}_K$. Then for $t<\tau_{R}$,  $N_2 \geq \sqrt{2\lambda_2/C}>0$,
we have from \eqref{eq:energy-est-2}, using the inverse Poicar\'{e} inequality as before, that
\begin{align*}
\frac{1}{2}  \frac{d}{dt} (\|\bfV\|^2 + \|\Tdiff\|^2) +&
 \frac{\prN}{2} \| \bfV \|^2 +
  \lambda_{2}\| \Tdiff\|^2   \\
  &\leq
    (\rab +  \prN \ra)\int_{\DD}  |\Tdiff v_d|dx + \frac{C}{\sqrt{\prN}} (\| \bfV\|^2\| \nabla\bfU\|^2
     + \left(\| \Tdiff\|^2+ \|\bfV\|^2\right)
  \|\nabla \Trs\|^2)  \\
 &\leq
  \frac{\prN}{4}\|\bfV\|^{2} + \frac{C(\rab +  \prN \ra)^2}{\prN}\|\Tdiff\|^{2} + \frac{C}{\sqrt{\prN}} (\| \bfV\|^2\| \nabla\bfU\|^2
     + \left(\| \Tdiff\|^2+ \|\bfV\|^2\right)
  \|\nabla \Trs\|^2).
\end{align*}
That is, with $\lambda_2 \geq 2C(\rab +  \prN \ra)^2/\prN$, taking $\lambda = (\frac{1}{2}\prN)\wedge \lambda_2$, we have
\begin{align}
\frac{d}{dt} (\|\bfV\|^2 + \|\Tdiff\|^2) +  \lambda\left( \| \bfV \|^2 +
 \| \Tdiff\|^2\right)
  & \leq   \frac{C}{\sqrt{\prN}} \left(\| \Tdiff\|^2+ \|\bfV\|^2\right)
  \left(\| \nabla\bfU\|^2 + \|\nabla \Trs\|^2\right)\,,
\label{eq:energy-est-10}
\end{align}
as in \eqref{eq:energy-est-5} above.  The rest of the proof follows
 the proof of part (i) above.  Note that, from \eqref{eq:def-c2} (specifically from the dependence of $C_2$ on $\prN$),
 we can choose $\prN$ large enough such that
$\lambda = (\frac{1}{2}\prN)\wedge \lambda_2 > \frac{ C C_2}{\prN^{3/2}}$, as required in the proof.

\end{proof}

\subsection{Proof of Theorem~\ref{thm:2}}

\begin{proof}[Proof of Theorem~\ref{thm:2}]  We write $H=H_2$ as the initial data phase space for \eqref{eq:vel:inf}--\eqref{eq:theta:inf}
during this proof.  We assume that the spatial dimension is $d=3$, as the proof when $d=2$ is nearly identical.
Let $P_\infty $  be the transition function on $H_{\infty}$ corresponding to equations \eqref{eq:vel:inf}--\eqref{eq:theta:inf} evaluated at integer times.
For given $U^0=(\Uo,\Tho)$ satisfying \eqref{eq:vel:inf}--\eqref{eq:theta:inf} and initial condition $\tilde{\Trs_0} \in H$, define $\tilde{U}=(\tilde{\bfU},\tilde{\Trs})$ as the solution of
\begin{align}
 - \Delta \tilde{\bfU} + \nabla \tilde{p} &=  \ra\hate \tilde{\Trs}, \quad \nabla \cdot \tilde{\bfU} = 0,\label{eq:mod:inf:vel} \\
  d \tilde{\Trs} + \tilde{\bfU} \cdot \nabla \tilde{\Trs} dt &= \rab \tilde{u}_d dt+  \Delta \tilde{\Trs} dt
  +  \sum_{k=1}^{N_2} \sigma_k dW^k - \lambda_2 P_{N_2} (\tilde{\Trs}- \Tho)\mathbbm{1}_{t<\tilde{\tau}_R}dt \,
  \label{eq:mod:inf:temp}
\end{align}
where
\begin{align}
\label{eq:stop:2}
\tilde{\tau}_R = \inf\left\{t>0:\int_{0}^{t} \|\sigma^{-1}\lambda_2 P_{N_2}(\tilde{\Trs}- \Tho)\|^2 ds \geq R
\right\}.
\end{align}
As in the proof of Theorem \ref{thm:1}, let $\tilde{P}_\infty$ be the transition function on $H_{\infty}$ corresponding to equations
\eqref{eq:mod:inf:vel} and \eqref{eq:mod:inf:temp}.
Again, by the Girsanov theorem, for any initial data $\tilde{\Trs}_0 \in H$,
one has $\Gamma_{\Tho_0, \tilde{\Trs}_0} \in
 \tilde{C}_{\infty}(\delta_{\Tho_0}, \delta_{\tilde{\Trs}_0})$
 with $\Gamma_{\Tho_0, \tilde{\Trs}_0} = \delta_{\Tho_0}P_\infty  \times
  \delta_{\tilde{\Trs}_0}\tilde{P}_\infty$.  Following the discussion in the proof of
 Theorem \ref{thm:1} and using Theorem \ref{thm:HMS2011},
 it remains to prove that $\Gamma_{\Tho_0, \tilde{\Trs}_0}(D)>0$.

For $\Vtot = \tilde{\Utot} - \Utot^0=(\bfV,\Tdiff)$, we obtain
\begin{align}
  & -\Delta \bfV  + \nabla p  =  \ra \hate \Tdiff,
  \label{eq:diff:inf:vel}\\
  &\pd{t} \Tdiff + \tilde{\bfU} \cdot \nabla \Tdiff + \bfV \cdot \nabla \Tho = v_d +\Delta \Tdiff
  - \lambda_2 P_{N_2} \Tdiff \mathbbm{1}_{t<\tilde{\tau}_R}.
  \label{eq:diff:inf:temp}
\end{align}
Energy estimates yield
\begin{align}
\|\nabla \bfV\|^2 &\leq \ra^2 \|\Tdiff\|^2 \,, \notag \\
   \frac{1}{2} \frac{d}{dt} \|\Tdiff\|^2+   \|\nabla \Tdiff\|^2 +
   \lambda_2 \|P_{N_2} \Tdiff\|^2\mathbbm{1}_{t<\tilde{\tau}_R}
   &\leq \int_{\DD} (\rab |v_d \Tdiff| + |\Tdiff\bfV \cdot \nabla \Tho|)dx,
\label{eq:inf:en}
\end{align}
and we find
\begin{align*}
\int_{\DD} |\Tdiff\bfV \cdot \nabla \Tho|dx
   &\leq   C\|\bfV\|_{L^6}\|\nabla \Tho\|\|\Tdiff\|_{L^3}  \leq   C\|\nabla \bfV\|\|\nabla \Tho\|\|\Tdiff\|^{1/2}\|\nabla\Tdiff\|^{1/2} \\
   &\leq   C\ra\|\Tdiff\|^{3/2}\|\nabla \Tho\|\|\nabla\Tdiff\|^{1/2} \leq   \frac{1}{2}\|\nabla \Tdiff\|^{2} + C \ra^{4/3}\|\Tdiff\|^{2}\|\nabla\Tho\|^{4/3} \\
   &\leq   \frac{1}{2}\|\nabla \Tdiff\|^{2} + C \ra^{4} \|\Tdiff\|^{2} + \frac{1}{2}\|\nabla\Tho\|^{2}\|\Tdiff\|^{2}.
\end{align*}
Combining with \eqref{eq:inf:en} we obtain
\begin{align}
 \frac{d}{dt} \|\Tdiff\|^2 + \|\nabla \Tdiff\|^2 +  2\lambda_2 \|P_{N_2} \Tdiff\|^2\mathbbm{1}_{t<\tilde{\tau}_R}
    \leq (C \ra^{4} + 2\ra\rab)\|\Tdiff\|^{2} + \|\nabla\Tho\|^{2}\|\Tdiff\|^{2}.
\label{eq:inf:en:2}
\end{align}
Let
$$
\tilde{A}_K=
\left\{\sup_{t\geq 0}\left(\|\Tho(t)\|^2 + \int_{0}^{t}
\|\nabla\Tho(s)\|^2 ds -  \tilde{C}_1-\tilde{C}_2t
\right) \leq \tilde{C}_3 K\right\}\,,
$$
where $\tilde{C}_1, \tilde{C}_2, \tilde{C}_3$ are the constants determined in Proposition \ref{lem:exp-bound-pr}.
With an appropriate choice of $R$, we will show that $\tilde{\tau}_R=\infty$ on $\tilde{A}_K.$
Indeed, assume $\tilde{\tau}_R$ is finite for some element of $\tilde{A}_K.$
Then for $t<\tilde{\tau}_R,$ by choosing $N_2 \geq \sqrt{\lambda_{2}/C}$, where $C$ is such that
$2\lambda_2 \|\Tdiff\|^2  \leq 2\lambda_2 \|P_{N_2}\Tdiff\|^2 + \frac{1}{2}\|\nabla\Tdiff\|^2$ (cf. \eqref{eq:ipo}),
 we obtain from
\eqref{eq:inf:en:2} that
\begin{align}
 \frac{d}{dt} \|\Tdiff\|^2+   \frac{1}{2}\|\nabla \Tdiff\|^2 +  2\lambda_2 \|\Tdiff\|^2
    \leq (C\ra^{4} + 2\ra\rab)\|\Tdiff\|^{2} + \|\nabla\Trs\|^{2}\|\Tdiff\|^{2}.
\label{eq:inf:en:3}
\end{align}
Now by selecting $\lambda_2> (C \ra^{4} + 2\ra\rab)\vee \tilde{C}_2 $ we find
\begin{align}
 \frac{d}{dt} \|\Tdiff\|^2+
   \lambda_2 \|\Tdiff\|^2
    \leq \|\nabla\Trs\|^{2}\|\Tdiff\|^{2},
\label{eq:inf:en:4}
\end{align}
and by Gr\"{o}nwall's inequality,
\begin{align*}
\|\Tdiff(t)\|^2 \leq  \|\Tdiff_{0}\|^2\exp\left(\int_{0}^{t}\|\nabla \Trs(s)\|^2ds-\lambda_2 t\right).
\end{align*}
Therefore, on $\tilde{A}_{K}$,
\begin{align}
\|\Tdiff(t)\|^2 \leq  \|\Tdiff_{0}\|^2\exp\left(\tilde{C}_3K +\tilde{C}_1\right)\exp\left((\tilde{C}_{2}-\lambda_2) t\right)\,.
\label{eq:Tbound6}
\end{align}
It follows that, for all $t<\tilde{\tau}_R$,
\begin{align*}
\int_{0}^{t} \|\sigma^{-1}\lambda _2 P_{N_2}\phi\|^2 ds &\leq
\lambda_2\tilde{C}\|\Tdiff_{0}\|^2\exp\left(\tilde{C}_{3}K+\tilde{C}_1\right)\int_{0}^{t}\exp\left((\tilde{C}_{2}-\lambda_2) s\right)ds
\\
&\leq \tilde{C}_0\|\Tdiff_{0}\|^2\exp\left(\tilde{C}_{3}K+\tilde{C}_1\right),
\end{align*}
where $\tilde{C}_0=\tilde{C}_0(\tilde{C}_{2})$.  Choose $R>0$ such that
$ R^2 > 2\tilde{C}_0\|\Tdiff_{0}\|^2
\exp\left(\tilde{C}_{3}K+C_1\right)$.
Then, since $\lambda_2  \geq \tilde{C}_2$ for all $0<t<\tilde{\tau}_R,$ we have
\begin{align}
\label{eq:attractor2}
\int_{0}^{t} \|\sigma^{-1}\lambda _2 P_{N_2}\phi\|^2 ds < \frac{R}{2},
\end{align}
which contradicts the assumption that $\tilde{\tau}_R$ is finite.  Therefore, for these choices of $K$, $R$ and $\lambda_2$, we have that on $\tilde{A}_{K}$, the estimate \eqref{eq:Tbound6} holds for all $t>0$,
and the remainder of the proof that $\Gamma_{\Tho_0, \tilde{\Trs}_0}(D)>0$ follows as in \eqref{eq:diag_estimate}--\eqref{eq:diag_2} above.
\end{proof}

\section{Bounds on the Nusselt Number}
\label{sec:5}

In this section we include the proof of Theorem \ref{thm:3}, which states that the Nusselt number $\Nu$ relative to the unique
invariant measure of \eqref{eq:B:eqn:vel}--\eqref{eq:B:eqn:temp} is observable, and provides quantitative bounds in terms of $\ra$, $\rab$.
Throughout this section, we assume $d=2$ and for a function $f=f(x_1,x_2,t)$ we use the notation
\begin{align}
\langle f\rangle = \langle f\rangle(x_1,x_2) = \lim_{t \rightarrow \infty}\frac{1}{t}\int_{0}^{t}f(t,x_1,x_2)dt
\label{eq:avg}
\end{align}
to denote the infinite temporal average.  We will also use the symbol
\begin{align*}
\overline{f}=\overline{f}(t,x_2)=\frac{1}{|\mathcal{D}|}\int_{0}^{L}f(t,x_1,x_2)dx_1
\end{align*}
 to refer to the horizontal average of $f$, and $\fint_{\mathcal{D}}dx$ for averages over $\mathcal{D}$.

\begin{proof}[Proof of Theorem \ref{thm:3}]
We assume $d=2$, $N_1=0$, $N_2=\infty$, and $\prN$ is sufficiently large such that, by Theorem \ref{thm:1}, the system \eqref{eq:B:eqn:vel}--\eqref{eq:B:eqn:temp}
possesses a unique ergodic invariant measure $\mu$.  Recall that the Nusselt number $\Nu$, relative to $\mu$, is given by \eqref{eq:nu:0}.
This notation is imprecise, and we should now interpret \eqref{eq:nu:0} according to \eqref{def:theta} as
\begin{align}
Nu = 1 + \int_{H}u_2 (\theta+\rab(1-x_2))d\mu(\bfU,\theta).
\end{align}
Next consider $f:H\rightarrow \RR$ given by
\begin{align*}
f(\tilde{\bfU},\tilde{\theta})=\int_{\mathcal{D}}\tilde{u}_2 (\tilde{\theta}+\rab(1-x_2))dx.
\end{align*}
Notice that from \eqref{eq:exp:your:mom:invar:measure} we have $f\in L^1(H;\mu)$.  Indeed,
\begin{align*}
\int_{H}\left(\int_{\mathcal{D}}\tilde{u}_2 (\tilde{\theta}+\rab(1-x_2))dx\right)d\mu(\tilde{\bfU},\tilde{\theta}) \leq
\frac{1}{2}\int_{H}\left(\|\tilde{\bfU}\|^{2} + \|\tilde{\theta}\|^{2} + \tilde{C}\right)d\mu(\tilde{\bfU},\tilde{\theta})<\infty,
\end{align*}
where $\tilde{C}=\tilde{C}(\rab)$.
By the Birkhoff ergodic theorem it follows that, for $\mu$ almost every initial condition
$(\bfU_0,\theta_0)\in H$, recalling \eqref{eq:nu:0},
\begin{align}
\label{eq:Tbound5}
\Nu &=  1 + \frac{1}{\rab}\int_{H}
\fint_{\mathcal{D}}\tilde{u}_2 (\tilde{\theta}+\rab(1-x_2))\, dx\,d\mu(\tilde{\bfU},\tilde{T})
 = 1 + \frac{1}{\rab}\left\langle\E\left(\fint_{\mathcal{D}}u_2 (\theta+\rab(1-x_2)) dx \right)\right\rangle  \notag \\
 &= 1 + \frac{1}{\rab}\left\langle\E\left(\fint_{\mathcal{D}}u_2 T dx \right)\right\rangle
\end{align}
where $(\bfU,T)$ is the solution to \eqref{eq:B:eqn:vel}--\eqref{eq:B:eqn:temp} with initial data $(\bfU_0,T_0)$
(note that $\theta_0$ and $T_0$ are related by \eqref{def:theta}).
It remains to establish part (ii) of Theorem \ref{thm:3}.

We follow the method of \cite{ConstantinDoering1996}.
Fix initial conditions $(\bfU_0,T_0)\in \text{supp}(\mu)$ such that \eqref{eq:Tbound5} is satisfied.
Notice that from \eqref{eq:B:eqn:temp} and integration by parts, recalling that $\|\sigma\|=1$, we have
\begin{align}
\label{eq:Tbound3}
d\|T\|^{2} &= \left(2\int_{\mathcal{D}}\Delta T Tdx  + 1\right)dt + 2\langle T,\sigma\rangle dW \notag \\
&=  \left(-2\|\nabla T\|^{2} - 2|\mathcal{D}|\rab \overline{\pd{2}T}|_{x_2=0} + 1\right)dt + 2\langle T,\sigma\rangle dW.
\end{align}
Also, by Proposition \ref{lem:exp-bound} (specifically \eqref{eq:exp:est:2}),
\begin{align}
\label{eq:Tbound4}
\frac{1}{t}\E (\|T(t)\|^{2} - \|T_0\|^{2}) \leq \frac{C}{t} \rightarrow 0,
\end{align}
as $t\rightarrow \infty$.
We then take the expectation of the integrated form of \eqref{eq:Tbound3}, and conclude by \eqref{eq:Tbound4} that
\begin{align}
\label{eq:nu:1}
\frac{1}{|\mathcal{D}|}\langle\E\|\nabla T\|^{2}\rangle = -\rab \langle\E\left(\overline{\pd{2}T}|_{x_2=0}\right)\rangle + \frac{1}{2|\mathcal{D}|}.
\end{align}
Next observe that \eqref{eq:B:eqn:temp} also gives
\begin{align*}
d\overline{T} = \left(\overline{\pd{2}^2 T}  -\overline{\pd{2}(u_2 T)}\right)dt + \overline{\sigma} dW,
\end{align*}
and therefore
\begin{align}
\label{eq:nu:2}
(1-x_2) d\overline{T}=(1-x_2)\left(\overline{\pd{2}^2 T} - \overline{\pd{2}(u_2 T})\right)dt + (1-x_2)\overline{\sigma} dW.
\end{align}
We now integrate \eqref{eq:nu:2} in $x_2$, and integrate by parts, to obtain
\begin{align}\label{eq:nu:3}
d\left(\int_{0}^{1}(1-x_2)\overline{T}dx_2 \right)
&=  \left(\int_{0}^{1} (1-x_2)\overline{\pd{2}^{2}T}dx_2 + \fint_{\mathcal{D}}(1-x_2)\pd{2}(u_2 T)dx\right)dt + \fint_{\mathcal{D}}(1-x_2)\sigma dW \notag \\
&=  \left( -\overline{\pd{2}T}|_{x_2=0} + \int_{0}^{1} \overline{\pd{2}T}dx_2 - \fint_{\mathcal{D}}u_2 T dx\right)dt + \fint_{\mathcal{D}}(1-x_2)\sigma dW \notag \\
&=  \left( -\overline{\pd{2}T}|_{x_2=0} - \rab - \fint_{\mathcal{D}}u_2 T dx\right)dt + \fint_{\mathcal{D}}(1-x_2)\sigma dW.
\end{align}
Note that we have used the boundary conditions \eqref{eq:bc} in the last line.
By taking the expectation and infinite temporal average, we have by Proposition \ref{lem:exp-bound}, another integration by parts, and
\eqref{eq:Tbound5},
\begin{align*}
 -\left\langle\E\left(\overline{\pd{2}T}|_{x_2=0}\right)\right\rangle - \rab
 = \left\langle\E\fint_{\mathcal{D}}u_2 T dx\right\rangle = \rab (\Nu - 1).
\end{align*}
Therefore
\begin{align*}
\rab\Nu =  -\left\langle\E\left(\overline{\pd{2}T}|_{x_2=0}\right)\right\rangle,
\end{align*}
and by \eqref{eq:nu:1},
\begin{align*}
\frac{1}{|\mathcal{D}|}\langle\E\|\nabla T\|^{2}\rangle = -\rab \left\langle\E\left(\overline{\pd{2}T}|_{x_2=0}\right)\right\rangle
+ \frac{1}{2|\mathcal{D}|} = \rab^2\Nu + \frac{1}{2|\mathcal{D}|},
\end{align*}
or equivalently,
\begin{align}
\label{eq:nu:4}
\Nu = \frac{1}{\rab^2|\mathcal{D}|}\langle\E\|\nabla T\|^{2}\rangle -  \frac{1}{2\rab^2|\mathcal{D}|}.
\end{align}
Next observe that \eqref{eq:B:eqn:vel} combined with Proposition \ref{lem:exp-bound} and \eqref{eq:Tbound5} gives
\begin{align*}
\frac{1}{|\mathcal{D}|}\left\langle\E \|\nabla \bfU\|^{2}\right\rangle = \ra \left\langle \E\fint_{\mathcal{D}}u_2 T \,dx\right\rangle = \ra\rab (\Nu - 1),
\end{align*}
so we also have
\begin{align}
\label{eq:nu:5}
\Nu = \frac{1}{\ra\rab|\mathcal{D}|}\left\langle\E\|\nabla \bfU\|^2\right\rangle + 1.
\end{align}
We multiply \eqref{eq:nu:5} by $-1/2$ and add it to \eqref{eq:nu:4} to obtain
\begin{align}
\label{eq:nu:6}
\Nu = \frac{2}{\rab^2|\mathcal{D}|}\langle\E\|\nabla T\|^{2}\rangle  -
\frac{1}{\ra\rab|\mathcal{D}|}\left\langle\E\|\nabla \bfU\|^2\right\rangle -  \frac{1}{\rab^2|\mathcal{D}|} - 1.
\end{align}

Consider a deterministic background profile $\tau = \tau(x_2)$ satisfying $\tau(0)=\rab$, $\tau(1)=0$ and consider $T(t,x)=\tau(x_2) + \theta(t,x)$.
Then the fluctuations
$\theta$ satisfy
\begin{align*}
d\theta + (\bfU \cdot \nabla \theta - \Delta \theta - \tau'' + u_2 \tau')dt = \sigma dW,
\end{align*}
with Dirichlet boundary conditions at $x_2=0,1$.  This gives
\begin{align*}
d\|\theta\|^{2} &= \left(2\int_{\mathcal{D}}(\Delta \theta \theta   + \tau''\theta  - u_2 \tau'\theta)dx + 1\right)dt + 2\langle T,\sigma\rangle
dW \\
&=  \left(-2\|\nabla \theta\|^{2} - 2\int_{\mathcal{D}}(\tau'\pd{2}\theta + u_2 \tau'\theta)dx + 1\right)dt + 2\langle T,\sigma\rangle
dW .
\end{align*}
By taking the expectation and infinite temporal average, we have again by Proposition \ref{lem:exp-bound},
\begin{align}
\label{eq:nu:7}
-\frac{1}{|\mathcal{D}|}\left\langle\E\|\nabla \theta\|^{2}\right\rangle
 - \left\langle\E\fint_{\mathcal{D}}(\tau'\pd{2}\theta + u_2 \tau'\theta)dx\right\rangle + \frac{1}{2|\mathcal{D}|}=0.
\end{align}
Next we expand
\begin{align}
\label{eq:nu:8}
\|\nabla T\|^{2} = \|\nabla \theta\|^{2} + 2\int_{\mathcal{D}}\tau' \pd{2}\theta dx + |\mathcal{D}|\int_{0}^{1}(\tau')^2 dx_2.
\end{align}
We take the expectation and temporal average of \eqref{eq:nu:8}, and add \eqref{eq:nu:7} twice to obtain
\begin{align}
\frac{1}{|\mathcal{D}|}\left\langle\E\|\nabla T\|^{2}\right\rangle = -
\frac{1}{|\mathcal{D}|}\left\langle\E\|\nabla \theta\|^{2}\right\rangle + \int_{0}^{1}(\tau')^2 dx_2
-2\left\langle\E\fint_{\mathcal{D}}u_2 \tau'\theta dx\right\rangle + \frac{1}{|\mathcal{D}|}.
\label{eq:nu:9}
\end{align}
Inserting \eqref{eq:nu:9} into \eqref{eq:nu:6} we find
\begin{align}
\label{eq:nu:10}
\Nu
&= \frac{2}{\rab^2}\int_{0}^{1}(\tau')^2 dx_2
+\frac{1}{\rab^2|\mathcal{D}|}-1 -\frac{2}{\rab^2|\mathcal{D}|}Q_{\tau}(\theta,\bfU),
\end{align}
where
\begin{align}
Q_{\tau}(\theta,\bfU) = \left\langle\E\left(\|\nabla \theta\|^{2} + \frac{\rab}{2\ra}\|\nabla \bfU\|^2 +
2\int_{\mathcal{D}}u_2 \tau'\theta dx\right)\right\rangle.
\label{eq:nu:11}
\end{align}

Consider a background temperature given by
\begin{align*}
\tau  = \tau(x_2)=\rab - \frac{\rab}{\delta}\int_{0}^{x_2}\psi(x/\delta)dx,
\end{align*}
where $\psi=\psi(x)$ is a smooth function on $\mathbb{R}$ such that $0\leq \psi\leq C$, $\text{supp}(\psi) \subset [0,1]$
and $\int_{\mathbb{R}}\psi(x)dx = 1$.  It follows that, for $\delta<1$,
$\tau(0)=\rab$ and $\tau(1)=0$.  Next, observe that for any $x_2\in[0,1]$, we can exploit the Dirichlet
boundary condition at $x_2=0$ and write
\begin{align*}
|\overline{u_2 \theta}(x_2) | &= \left|\int_{0}^{x_2}\overline{u_2 \theta}'(z)dz  \right|
=  \frac{1}{|\mathcal{D}|}\left|\int_{0}^{L}\int_{0}^{x_2}\left( \theta\pd{2}u_2 + u_2  \pd{2}\theta\right)(x_1,z)dz dx_1\right|
\notag \\
&=  \frac{1}{|\mathcal{D}|}\left|\int_{0}^{L}\int_{0}^{x_2}\left[\left(\int_{0}^{z}\pd{2}\theta(x_1,y)dy\right)\pd{2}u_2(x_1,z)
+ \left(\int_{0}^{z}\pd{2}u_2(x_1,y)dy\right) \pd{2}
\theta(x_1,z)\right] dz dx_1\right| \notag \\
&\leq |x_2|\|\nabla u\|\|\nabla \theta\|,
\end{align*}
where we have applied Cauchy-Schwarz in each variable to obtain the last line.
It follows that
\begin{align*}
\left|\int_{\mathcal{D}}u_2 \tau'\theta dx\right| &= \left|\int_{0}^{1}\overline{u_2 \theta}(x_2)\tau'(x_2)dx_2\right|
\leq\rab\|\nabla u\|\|\nabla \theta\|\int_{0}^{1}|x_2/\delta|\psi(x_2/\delta)dx_2 \\
&= \rab\|\nabla u\|\|\nabla \theta\|\delta\int_{0}^{1}|z|\psi(z)dz \leq  \sqrt{\rab/2\ra}\|\nabla u\|\|\nabla \theta\|,
\end{align*}
by taking $\delta \sim 1/\sqrt{2\rab\ra}$.  With this choice for $\tau$, by \eqref{eq:nu:11} we have
\begin{align*}
Q_{\tau}(\theta,\bfU) \geq
\left\langle\E\left[\left(\|\nabla \theta\| - \sqrt{\rab/2\ra}\|\nabla \bfU\|\right)^{2}\right]\right\rangle \geq 0,
\end{align*}
for all $(\theta,\bfU)\in H$, and therefore
\begin{align}
\label{eq:nu:bound}
\Nu &\leq \frac{2}{\rab^2}\int_{0}^{1}(\tau')^2 dx_2 +\frac{1}{\rab^2|\mathcal{D}|}-1
= \frac{2}{\delta^2}\int_{0}^{1}(\psi(x_2/\delta))^2 dx_2 +\frac{1}{\rab^2|\mathcal{D}|}-1 \notag \\
&= \frac{2}{\delta}\int_{0}^{1}(\psi(z))^2 dz +\frac{1}{\rab^2|\mathcal{D}|}-1 \notag \\
&\leq  C\sqrt{\rab\ra} +\frac{1}{\rab^2|\mathcal{D}|}-1.
\end{align}
\end{proof}

\begin{Rmk}
Notice that, due to the stochastic forcing, our bound on the Nusselt number in \eqref{eq:nu:bound}
includes an ``It\^{o} correction'' term of the form $\frac{1}{\rab^2|\mathcal{D}|}$.  Further note that,
from \eqref{eq:phys}, by keeping all other parameters fixed and taking the stochastic heat flux $\gamma$ large,
$\ra\rab$ is constant and $\rab$ goes to zero, so that this extra term dominates the right-hand side of \eqref{eq:nu:bound}, and
 we have weaker estimates on the Nusselt number $Nu$.
 We do not claim that this reflects an enhanced convective heat transport
 for \eqref{eq:B:eqn:vel}--\eqref{eq:B:eqn:temp} as $\gamma$ increases,
 but rather that our mathematical methods produce an extra term in this case, which
may or may not be misleading.
\end{Rmk}

\appendix

\section{Appendix: Non-dimensionalization of the Equations}
\label{sec:non-dim:equations}

In this appendix we will describe how to obtain the non-dimensional system \eqref{eq:B:eqn:vel}--\eqref{eq:B:eqn:temp} from
the standard form of these equations through a rescaling.  We also derive the system considered in \cite{FoldesGlattHoltzRichardsThomann2014b}.
The Boussinesq equations in their original form can be written (with
stochastic forcing included) as
\begin{align}
&d\vstar + \left(\vstar \cdot \nabla \vstar + \nabla \pstar\right)d\tstar = (g\astar\hate \Tstar  + \nu \Delta \vstar)d\tstar
+ \tilde{\gamma}\sum_{k=1}^{N_1}\sigmatildestar_k d\tildeWstar^k \,, \quad \nabla \cdot \vstar = 0,
\label{eq:Vel:Physical:short}\\
&d\Tstar + \left( \vstar\cdot\nabla \Tstar \right)d\tstar = \kappa \Delta \Tstar d\tstar + \gamma\sum_{k=1}^{N_2}\sigmastar_k d\Wstar^k .
\label{eq:Temp:Physical:short}
\end{align}
We denote the spatial variable by $\mathbf{x}^* = (x_1^*,\ldots,x_d^*)$, where $d=2$ or $d=3$, assume that $\vstar = (v_1,\ldots,v_d)$
and $\Tstar$ are periodic in $(x^*_1,\ldots,x^*_{d-1})$, and that
$$
\vstar|_{x_d^*=0} = \vstar|_{x_d^*=h} = 0,~~~
\Tstar|_{x_d^*=0}=T_1,~~ \Tstar|_{x_d^*=h}=0,$$
where $T_1>0.$
In these equations, $g$ represents
the gravitational acceleration, $\nu$ the kinematic viscosity, $\kappa$
the thermal conductivity, $\astar$ the thermal expansion coefficient,
and $\tilde{\gamma},\gamma$ the volumetric flux coefficients, which will be specified in more detail
depending on context below.
Here $\sigmatildestar_k$ and $\sigmastar_k$ denote bases of
eigenfunctions of the (time independent) Stokes and Laplace operators, respectively, with appropriate
boundary conditions, and $\tildeWstar^k,\Wstar^k$ are sequences of one-dimensional,
mutually independent, standard Brownian motions.

We abuse notation for computational purposes, and rewrite \eqref{eq:Vel:Physical:short}-\eqref{eq:Temp:Physical:short} using standard derivative notation
\begin{align}
&\frac{\partial\vstar}{\partial \tstar} + \vstar \cdot \nabla \vstar + \nabla \pstar = g\astar\hate \Tstar  + \nu \Delta \vstar
+\tilde{\gamma}\sum_{k=1}^{N_1}\sigmatildestar_k \frac{\partial\tildeWstar^k}{\partial \tstar},
\label{eq:Vel:Physical:ab}\\
&\frac{\partial \Tstar}{\partial \tstar} +  \vstar\cdot\nabla \Tstar  = \kappa \Delta \Tstar
+  \gamma\sum_{k}^{N_2}\sigmastar_k
\frac{\partial\Wstar^k}{\partial \tstar}
.
\label{eq:Temp:Physical:ab}
\end{align}
To obtain the non-dimensional version of \eqref{eq:Vel:Physical:ab}-\eqref{eq:Temp:Physical:ab}, we change variables as follows
\begin{align*}
\mathbf{x}^* &= h\mathbf{x}, \\
\tstar &= \beta t, \\
\vstar(\mathbf{x}^*,\tstar) &= \lambda \bfU(\mathbf{x},t), \\
\Tstar(\mathbf{x}^*,\tstar) &= \tilde{T} \T(\mathbf{x},t), \\
\pstar(\mathbf{x}^*,\tstar) &= \frac{\nu\lambda}{h}p(\mathbf{x}, t)\,,
\end{align*}
where $\beta = h^{2}/\kappa$ and $\lambda = \kappa /h$.  We compute
\begin{align*}
&\frac{\lambda}{\beta}\frac{\partial \bfU}{\partial t} + \frac{\lambda^{2}}{h}\bfU \cdot \nabla \bfU + \frac{\nu \lambda}{h^{2}}\nabla p
= g\astar \tilde{T} \hate \T + \frac{\nu \lambda}{h^2}  \Delta \bfU
+\frac{\tilde{\gamma}}{h^{d/2}\sqrt{\beta}}\sum_{k=1}^{N_1}\tilde{\sigma}_k \frac{\partial\tilde{W}^k}{\partial t},\\
&\frac{\tilde{T}}{\beta}\frac{\partial \T}{\partial t} +  \frac{\lambda \tilde{T}}{h}\bfU \cdot\nabla \T  = \frac{\tilde{T}\kappa}{h^2} \Delta T
+  \frac{\gamma}{h^{d/2}\sqrt{\beta}} \sum_{k=1}^{N_2}\sigma_k\frac{\partial W^k}{\partial t},
\end{align*}
where we have used that $\frac{1}{\beta^{1/2} }\Wstar^k(\beta t) =:W^k(t)$ is
a standard Brownian motion, and $\|\sigmastar_k\| = h^{-d/2}
\|\sigma_k\|$.

Notice that $\frac{\tilde{T}}{\beta} = \frac{\lambda \tilde{T}}{h} = \frac{\tilde{T} \kappa}{h^2}$ and $\frac{\lambda}{\beta}=\frac{\lambda^2}{h}$,
and if we recall $\prN = \nu/\kappa$, we obtain
\begin{align}
&\frac{1}{\prN}\left(\frac{\partial \bfU}{\partial t} + \bfU \cdot \nabla \bfU\right) + \nabla p =  \frac{g \alpha \tilde{T}h^3}{\kappa \nu}\hate \T
+  \Delta \bfU + \frac{\tilde{\gamma}}{\nu \sqrt{\kappa}
h^{\frac{d}{2} -2}}\sum_{k=1}^{N_1}\tilde{\sigma}_k \frac{\partial\tilde{W}^k}{\partial t},
\label{eq:B:vel:nodim}\\
&\frac{\partial \T}{\partial t} +  \bfU \cdot\nabla \T  =  \Delta T +
\frac{\gamma}{\tilde{T}\sqrt{\kappa} h^{d/2-1}} \sum_{k=1}^{N_2}\sigma_k\frac{\partial W^k}{\partial t}.
\label{eq:B:temp:nodim}
\end{align}
with the boundary conditions
$$
\bfU|_{x_d=0} = \bfU|_{x_d=h} = 0,~~~
\T|_{x_d=0}=\frac{T_1}{\tilde{T}},~~ \T|_{x_d=h}=0.
$$

The non-dimensionalization proceeds as follows: we select $\beta = h^{2}/\kappa$ as the unit of time, $h$ the unit of space, and further select
$\tilde{T} = \frac{\gamma\sqrt{\beta}}{h^{d/2}}=\frac{\gamma}{\sqrt{\kappa}h^{d/2-1}}$ as the unit of temperature.
We also absorb the constant $\frac{\tilde{\gamma}}{\nu \sqrt{\kappa}
h^{\frac{d}{2} -2}}$ into the definition of the basis $\{\tilde{\sigma_k}\}$, leading to the system \eqref{eq:B:eqn:vel}--\eqref{eq:B:eqn:temp}.
The parameters in the problem are then the unit-less ``Rayleigh numbers''
$\ra := \frac{g \alpha \tilde{T}h^3}{\kappa \nu} = \frac{g\alpha \gamma h^{4-d/2}}{\nu \kappa^{3/2}}$
and
$ \rab: = \frac{T_1}{\tilde{T}} = \frac{\sqrt{\kappa}h^{d/2-1} T_1}{\gamma}$.
Here we have defined the units of $\gamma$ and $\tilde{\gamma}$ as required to maintain consistency and reach a non-dimensional form for
\eqref{eq:B:eqn:vel}--\eqref{eq:B:eqn:temp}.

For example, in the work \cite{FoldesGlattHoltzRichardsThomann2014b} we consider $d=3$, $\tilde{\gamma}=0$, and interpret $\gamma=\frac{H}{\rho c}$ as the volumetric ``stochastic''
heat flux $H$ (units of power$/\sqrt{\text{volume*time}}$) normalized by the density $\rho$
and the specific heat $c$ of the fluid.  Then
\eqref{eq:B:vel:nodim}--\eqref{eq:B:temp:nodim} becomes
\begin{align}
&\frac{1}{\prN}\left(\frac{\partial \bfU}{\partial t} + \bfU \cdot \nabla \bfU\right) + \nabla p =  \ra\hate \T
+  \Delta \bfU,
\label{eq:B:vel:nodim:2}\\
&\frac{\partial \T}{\partial t} +  \bfU \cdot\nabla \T  =  \Delta T + \sum_{k=1}^{N_2}\sigma_k\frac{\partial W^k}{\partial t},
\label{eq:B:temp:nodim:2}
\end{align}
with the boundary conditions
$$
\bfU|_{x_d=0} = \bfU|_{x_d=1} = 0,~~~
\T|_{x_d=0}=\rab,~~ \T|_{x_d=1}=0.
$$

\section*{Acknowledgments}

A large portion of this work was carried out during several research fellowships.
The project was initiated during a Research in Peace fellowship at the
Institut Mittag-Leffler, Stockholm, Sweden, in July 2014.  It continued during a Research in Pairs fellowship at the
Mathematical Research Institute of Oberwolfach, Germany,
 in August 2015.  The details of the manuscript were finalized when the authors visited the Mathematical
Sciences Research Institute in Berkeley, California, in September 2015, and is therefore supported in part by the National
Science Foundation under grant No. 0932078 000.
This work continues a collaboration which began during a research visit  at
the Institute for Mathematics and its Applications in Minneapolis, Minnesota, over 2012-2013 academic year.   We are grateful to all of these wonderful institutions for supporting this work.
We would also like to thank C. Doering, J. Mattingly and E. Thomann for inspiring discussions.
NEGH was partially supported by the National Science Foundation under the grant
NSF-DMS-1313272.

\begin{footnotesize}
\bibliographystyle{alpha}

\begin{thebibliography}{FGHRT15}

\bibitem[AGL09]{AhlersGrossmannLohse2009}
G.~Ahlers, S.~Grossmann, and D.~Lohse.
\newblock Heat transfer and large scale dynamics in turbulent
  rayleigh-b{\'e}nard convection.
\newblock {\em Reviews of Modern Physics}, 81(2):503, 2009.

\bibitem[B\'01]{Benard1900}
M.~B\'{e}nard.
\newblock Les tourbillons cellulaires dans une nappe liquide propageant de la
  chaleur par convection en r\'{e}gime permanent.
\newblock {\em Ann. d. Chimie et de Physique}, xxiii:62--144, 1901.

\bibitem[Ben95]{Bensoussan1}
A.~Bensoussan.
\newblock Stochastic {N}avier-{S}tokes equations.
\newblock {\em Acta Appl. Math.}, 38(3):267--304, 1995.

\bibitem[BH09]{BreuerHansen09}
M.~Breuer and U.~Hansen.
\newblock Turbulent convection in the zero reynolds number limit.
\newblock {\em Euro. phys. Lett.}, 86(24004), 2009.

\bibitem[BN12]{BarlettaNield2012}
A.~Barletta and D.~A. Nield.
\newblock On the rayleigh--b{\'e}nard--poiseuille problem with internal heat
  generation.
\newblock {\em International Journal of Thermal Sciences}, 57:1--16, 2012.

\bibitem[Bou97]{Boussinesq1897}
J.V. Boussinesq.
\newblock Th\'eorie de l'\'ecoulement tourbillonnant et tumultueux des liquides
  dans les lits rectilignes a grande section.
\newblock {\em Des comptes rendus des s\'eances des sciences}, 1897.

\bibitem[BPA00]{BodenschatzPeschAhlers2000}
E.~Bodenschatz, W.~Pesch, and G.~Ahlers.
\newblock Recent developments in {R}ayleigh-{B}{\'e}nard convection.
\newblock In {\em Annual review of fluid mechanics, {V}ol. 32}, volume~32 of
  {\em Annu. Rev. Fluid Mech.}, pages 709--778. Annual Reviews, Palo Alto, CA,
  2000.

\bibitem[Bus70]{busse1970}
F.H. Busse.
\newblock Bounds for turbulent shear flow.
\newblock {\em Journal of Fluid Mechanics}, 41(01):219--240, 1970.

\bibitem[Bus78]{busse1978}
F.~H. Busse.
\newblock The optimum theory of turbulence.
\newblock {\em Advances in applied mechanics}, 18:77--121, 1978.

\bibitem[CD96]{ConstantinDoering1996}
P.~Constantin and C.~R. Doering.
\newblock Heat transfer in convective turbulence.
\newblock {\em Nonlinearity}, 9(4):1049--1060, 1996.

\bibitem[CD99]{ConstantinDoerin1999}
P.~Constantin and C.~R. Doering.
\newblock Infinite {P}randtl number convection.
\newblock {\em J. Statist. Phys.}, 94(1-2):159--172, 1999.

\bibitem[CF88]{ConstantinFoias88}
P.~Constantin and C.~Foias.
\newblock {\em Navier-{S}tokes equations}.
\newblock Chicago Lectures in Mathematics. University of Chicago Press,
  Chicago, IL, 1988.

\bibitem[DC01]{ConstantinDoering2001}
C.~R. Doering and P.~Constantin.
\newblock On upper bounds for infinite {P}randtl number convection with or
  without rotation.
\newblock {\em J. Math. Phys.}, 42(2):784--795, 2001.

\bibitem[DGHT11]{DebusscheGlattHoltzTemam1}
A.~Debussche, N.~Glatt-Holtz, and R.~Temam.
\newblock Local martingale and pathwise solutions for an abstract fluids model.
\newblock {\em Physica D}, 240(14-15):1123--1144, 2011.

\bibitem[DPZ92]{ZabczykDaPrato1992}
G.~Da~Prato and J.~Zabczyk.
\newblock {\em Stochastic equations in infinite dimensions}, volume~44 of {\em
  Encyclopedia of Mathematics and its Applications}.
\newblock Cambridge University Press, 1992.

\bibitem[Eyi96]{Eyink96}
G.L. Eyink.
\newblock Exact results on stationary turbulence in 2d: consequences of
  vorticity conservation.
\newblock {\em Physica D: Nonlinear Phenomena}, 91(1):97--142, 1996.

\bibitem[FG95]{FlandoliGatarek1}
F.~Flandoli and D.~Gatarek.
\newblock Martingale and stationary solutions for stochastic {N}avier-{S}tokes
  equations.
\newblock {\em Probab. Theory Related Fields}, 102(3):367--391, 1995.

\bibitem[FGHR]{FoldesGlattHoltzRichardsThomann2014b}
J.~F{\"o}ldes, N.~Glatt-Holtz, and G.~Richards.
\newblock Large prandtl number asymptotics in randomly forced turbulent
  convection.
\newblock submitted.

\bibitem[FGHRT15]{FoldesGlattHoltzRichardsThomann2013}
J.~F{\"o}ldes, N.~Glatt-Holtz, G.~Richards, and E.~Thomann.
\newblock Ergodic and mixing properties of the boussinesq equations with a
  degenerate random forcing.
\newblock {\em Journal of Functional Analysis}, 269(8):2427 -- 2504, 2015.

\bibitem[FMRT01]{FoiasManleyRosaTemam01}
C.~Foias, O.~Manley, R.~Rosa, and R.~Temam.
\newblock {\em Navier-{S}tokes equations and turbulence}, volume~83 of {\em
  Encyclopedia of Mathematics and its Applications}.
\newblock Cambridge University Press, Cambridge, 2001.

\bibitem[GHMR]{GlattHoltzMattinglyRichards2015}
N.~Glatt-Holtz, J.~Mattingly, and G.~Richards.
\newblock A simple framework for the analysis of ergodicity in infinite
  dimensional stochastic systems.
\newblock (to appear).

\bibitem[GH{\v S}V15]{GlattHoltzSverakVicol2013}
N.~Glatt-Holtz, V.~{\v S}ver{\' a}k, and V.~Vicol.
\newblock On inviscid limits for the stochastic navier–stokes equations and
  related models.
\newblock {\em Archive for Rational Mechanics and Analysis}, 217(2):619--649,
  2015.

\bibitem[GHZ09]{GlattHoltzZiane2}
N.~Glatt-Holtz and M.~Ziane.
\newblock Strong pathwise solutions of the stochastic {N}avier-{S}tokes system.
\newblock {\em Advances in Differential Equations}, 14(5-6):567--600, 2009.

\bibitem[GK96]{gyongy1996}
I.~Gy{\"o}ngy and N.~Krylov.
\newblock Existence of strong solutions for it{\^o}'s stochastic equations via
  approximations.
\newblock {\em Probability theory and related fields}, 105(2):143--158, 1996.

\bibitem[GS12]{GoluskinSpiegel2012}
D.~Goluskin and E.~A. Spiegel.
\newblock Convection driven by internal heating.
\newblock {\em Physics Letters A}, 377(1):83--92, 2012.

\bibitem[HM06]{HairerMattingly06}
M.~Hairer and J.~C. Mattingly.
\newblock Ergodicity of the 2{D} {N}avier-{S}tokes equations with degenerate
  stochastic forcing.
\newblock {\em Ann. of Math. (2)}, 164(3):993--1032, 2006.

\bibitem[HM08]{HairerMattingly2008}
M.~Hairer and J.~C. Mattingly.
\newblock Spectral gaps in {W}asserstein distances and the 2{D} stochastic
  {N}avier-{S}tokes equations.
\newblock {\em Ann. Probab.}, 36(6):2050--2091, 2008.

\bibitem[HM11]{HairerMattingly2011}
M.~Hairer and J.~C. Mattingly.
\newblock A theory of hypoellipticity and unique ergodicity for semilinear
  stochastic pdes.
\newblock {\em Electron. J. Probab.}, 16(23):658--738, 2011.

\bibitem[HMS11]{HairerMattinglyScheutzow2011}
M.~Hairer, J.~C. Mattingly, and M.~Scheutzow.
\newblock Asymptotic coupling and a general form of {H}arris' theorem with
  applications to stochastic delay equations.
\newblock {\em Probab. Theory Related Fields}, 149(1-2):223--259, 2011.

\bibitem[Hop40]{hopf1940}
E.~Hopf.
\newblock Ein allgemeiner endlichkeitssatz der hydrodynamik.
\newblock {\em Mathematische Annalen}, 117(1):764--775, 1940.

\bibitem[How72]{howard1972}
L.~N. Howard.
\newblock Bounds on flow quantities.
\newblock {\em Annual Review of Fluid Mechanics}, 4(1):473--494, 1972.

\bibitem[HS12]{hofmanova2012}
M.~Hofmanov{\'a} and J.~Seidler.
\newblock On weak solutions of stochastic differential equations.
\newblock {\em Stochastic Analysis and Applications}, 30(1):100--121, 2012.

\bibitem[KB37]{KryloffBogoliouboff1937}
N.~Kryloff and N.~Bogoliouboff.
\newblock La th{\'e}orie g{\'e}n{\'e}rale de la mesure dans son application
  {\`a} l'{\'e}tude des syst{\`e}mes dynamiques de la m{\'e}canique non
  lin{\'e}aire.
\newblock {\em Ann. of Math. (2)}, 38(1):65--113, 1937.

\bibitem[Kry10]{Krylov2010}
N.~V. Krylov.
\newblock It\^o's formula for the {$L_p$}-norm of stochastic {$W^1_p$}-valued
  processes.
\newblock {\em Probab. Theory Related Fields}, 147(3-4):583--605, 2010.

\bibitem[KS12]{KuksinShirikian12}
S.~Kuksin and A.~Shirikyan.
\newblock {\em Mathematics of Two-Dimensional Turbulence}.
\newblock Number 194 in Cambridge Tracts in Mathematics. Cambridge University
  Press, 2012.

\bibitem[LDB04]{LuDoeringBusse2004}
L.~Lu, C.~R. Doering, and F.~H. Busse.
\newblock {Bounds on convection driven by internal heating}.
\newblock {\em Journal of Mathematical Physics}, 45(7):2967--2986, 2004.

\bibitem[LR16]{Rayleigh1916}
O.M.F.R.S. Lord~Rayleigh.
\newblock On convection currents in a horizontal layer of fluid, when the
  higher temperature is on the under side.
\newblock {\em Philosophical Magazine Series 6}, 32(Issue 192):529--546, 1916.

\bibitem[LX10]{LohseXia2010}
D.~Lohse and K.~Q. Xia.
\newblock Small-scale properties of turbulent rayleigh-b{\'e}nard convection.
\newblock {\em Annual Review of Fluid Mechanics}, 42:335--364, 2010.

\bibitem[Mal54]{malkus1954}
W.~V.~R. Malkus.
\newblock The heat transport and spectrum of thermal turbulence.
\newblock {\em Proceedings of the Royal Society of London A: Mathematical,
  Physical and Engineering Sciences}, 225(1161):196--212, 1954.

\bibitem[Man06]{Manneville2006}
P.~Manneville.
\newblock Rayleigh-b{\'e}nard convection: Thirty years of experimental,
  theoretical, and modeling work.
\newblock In {\em Dynamics of Spatio-Temporal Cellular Structures}, pages
  41--65. Springer, 2006.

\bibitem[Nov65]{Novikov1965}
E.~A. Novikov.
\newblock Functionals and the random-force method in turbulence theory.
\newblock {\em Soviet Physics JETP}, 20:1290--1294, 1965.

\bibitem[OS11]{OttoSeis2011}
F.~Otto and C.~Seis.
\newblock Rayleigh--b{\'e}nard convection: improved bounds on the nusselt
  number.
\newblock {\em Journal of mathematical physics}, 52(8):083702, 2011.

\bibitem[Par06]{park2006}
J.~Park.
\newblock Dynamic bifurcation theory of rayleigh-benard convection with
  infinite prandtl number.
\newblock {\em Discrete Contin. Dyn. Syst.}, 6(3):591, 2006.

\bibitem[Rob67]{Roberts1967}
P.~H. Roberts.
\newblock Convection in horizontal layers with internal heat generation.
  theory.
\newblock {\em Journal of Fluid Mechanics}, 30(01):33--49, 1967.

\bibitem[RY99]{RevuzYor1999}
D.~Revuz and M.~Yor.
\newblock {\em Continuous martingales and {B}rownian motion}, volume 293 of
  {\em Grundlehren der Mathematischen Wissenschaften [Fundamental Principles of
  Mathematical Sciences]}.
\newblock Springer-Verlag, Berlin, third edition, 1999.

\bibitem[Sta88]{Stanisic1985}
M.~M. Stani{\v{s}}i{\'c}.
\newblock {\em The mathematical theory of turbulence}.
\newblock Universitext. Springer-Verlag, New York, second edition, 1988.

\bibitem[Tem01]{Temam2001}
R.~Temam.
\newblock {\em Navier-{S}tokes equations: Theory and numerical analysis}.
\newblock AMS Chelsea Publishing, Providence, RI, 2001.
\newblock Reprint of the 1984 edition.

\bibitem[TZ67]{TrittonZarraga1967}
D.~J. Tritton and M.~N. Zarraga.
\newblock Convection in horizontal layers with internal heat generation.
  experiments.
\newblock {\em Journal of Fluid Mechanics}, 30(01):21--31, 1967.

\bibitem[VCK12]{VenturiChoiKarniadakis12}
D.~Venturi, M.~Choi, and G.E. Karniadakis.
\newblock Supercritical quasi-conduction states in stochastic
  rayleigh–b\'{e}nard convection.
\newblock {\em Int. J. Heat Mass Transfer}, 55(13--14):3732--3743, 2012.

\bibitem[VKF79]{VishikKomechFusikov1979}
M.~I. Vishik, A.~I. Komech, and A.~V. Fursikov.
\newblock Some mathematical problems of statistical hydromechanics.
\newblock {\em Uspekhi Mat. Nauk}, 34(5(209)):135--210, 256, 1979.

\bibitem[VWK10]{VenturiWanKarniadakis10}
D.~Venturi, X.~Wan, and G.E. Karniadakis.
\newblock Stochastic bifurcation analysis of rayleigh–b\'{e}nard convection.
\newblock {\em J. Fluid Mech.}, 650:391--413, 2010.

\bibitem[Wan04]{Wang2004a}
X.~Wang.
\newblock Infinite {P}randtl number limit of {R}ayleigh-{B}{\'e}nard
  convection.
\newblock {\em Comm. Pure Appl. Math.}, 57(10):1265--1282, 2004.

\bibitem[Wan08]{Wang2008b}
X.~Wang.
\newblock Stationary statistical properties of {R}ayleigh-{B}{\'e}nard
  convection at large {P}randtl number.
\newblock {\em Comm. Pure Appl. Math.}, 61(6):789--815, 2008.

\bibitem[WD11]{WhiteheadDoering2011}
J.~P. Whitehead and C.~R. Doering.
\newblock Internal heating driven convection at infinite prandtl number.
\newblock {\em Journal of Mathematical Physics}, 52(9):093101, 2011.

\end{thebibliography}

\end{footnotesize}

\end{document}